\documentclass{amsart}

\usepackage{amsmath}
\usepackage{amsfonts}

\usepackage{verbatim}

\usepackage{amssymb,amsthm,amscd}
\usepackage[OT4]{fontenc}

\usepackage{hyperref}
\usepackage[pageref]{backref}

\renewcommand*{\backref}[1]{}
\renewcommand*{\backrefalt}[4]{%
    \ifcase #1 (Not cited.)%
    \or        (Cited on page~#2.)%
    \else      (Cited on pages~#2.)%
    \fi}

\newtheorem{thm}{Theorem}
\newtheorem{lem}[thm]{Lemma}

\newtheorem{cor}[thm]{Corollary}
\newtheorem{observation}[thm]{Remark}

\newtheorem{prop}[thm]{Proposition}
\newtheorem{defn}[thm]{Definition}

\numberwithin{thm}{section}
\numberwithin{equation}{section}

\DeclareMathOperator\bR{\mathbb{R}}
\DeclareMathOperator\bC{\mathbb{C}}

\DeclareMathOperator\bN{\mathbb{N}}

\newcommand{\cA}{\mathcal{A}}

\newcommand{\cE}{\mathcal{E}}

\newcommand{\cF}{\mathcal{F}}

\renewcommand{\d}{\partial}
\newcommand{\ddbar}{\d\overline{\d}}
\newcommand{\vepsilon}{\varepsilon}
\newcommand{\vphi}{\varphi}
\newcommand{\ii}{\sqrt{-1}}

\title[Weak solutions on compact Hermitian manifolds]{Weak solutions of complex Hessian equations on compact Hermitian manifolds}
\author[S. Ko\l odziej and N.-C. Nguyen]{S\l awomir Ko\l odziej and Ngoc Cuong Nguyen} 

\begin{document}

\begin{abstract}
 We prove the existence of weak solutions of complex $m-$Hessian equations on compact Hermitian manifolds
for the nonnegative  right hand side belonging to $L^p, p>n/m$ ($n$ is the dimension of the manifold). For smooth, positive data the equation has been recently
solved by Sz\'ekelyhidi and Zhang. We also give a stability result for such solutions.
\end{abstract}

\maketitle

\section{Introduction}

S.-T. Yau \cite{Y} confirmed the Calabi  Conjecture solving the complex  Monge-Amp\`ere on compact K\"ahler manifolds.
This fundamental result has been extended in several  directions. One can consider weak solutions for possibly degenerate non-smooth right hand side (see \cite{kolodziej98}).
Then,  one can generalize the equation, and here the Hessian equations are
a natural choice. The solutions were obtained by Dinew and the first author \cite{dinew-kolodziej12, dinew-kolodziej14}. One can also drop the K\"ahler condition and consider just Hermitian manifolds. 
The Monge-Amp\`ere on compact Hermitian manifolds was solved by Tosatti and  Weinkove \cite{TW10b}  for smooth nondegenerate data
and by the authors \cite{KN1} for the nonnegative  right hand side in $L^p$, $p>1$.
Very recently  Sz\'ekelyhidi \cite{szekelyhidi15} and Zhang \cite{dzhang15} showed the counterpart of Calabi-Yau theorem for Hessian equations on  compact Hermitian manifolds.

As in the real case geometrically meaningful Hessian equations appear in some ''twisted'' nonstandard form. Thus, for the K\"ahler manifolds the Fu-Yau equation
\cite{FY08} related to a Strominger system for dimension higher than two becomes the Hessian (two) equation with an extra linear term involving the gradient
of the solution. It has been recently studied by Phong-Picard-Zhang \cite{PPZ}. Another form  of the Hessian equation is shown to be equivalent to
quaternionic Monge-Amp\`ere equation on  HKT-manifolds in the paper of Alesker and Verbitsky \cite{AV}. Some related equations are solved by
Sz\'ekelyhidi-Tosatti-Weinkove in their work on the Gauduchon conjecture \cite{SzTW}.

The main result of this paper extends the  Sz\'ekelyhidi-Zhang \cite{szekelyhidi15, dzhang15} theorem as follows.

\noindent {\bf Theorem.} {\em Let $(X,\omega)$ be a compact $n$-dimensional Hermitian manifold and an integer number $1\leq m < n$. Let  $0 \leq f \in L^p(X, \omega^n), p>n/m$, and $\int_X f \omega^n >0$. There exist a continuous  $(\omega, m)$-subharmonic function $u$ and a constant $c>0$ satisfying
\[
	(\omega+ dd^c u)^m\wedge \omega^{n-m} = c f \omega^n.
\]}
We also obtain a stability theorem (Prop.~\ref{stability2}), which for the Monge-Amp\`ere equation was proven in \cite{KN2}. To obtain those
results we need to adapt the methods of pluripotential theory to Hessian equations and Hermitian setting. One of the key points, which required a different proof  was
the counterpart of Chern-Levine-Nirenberg inequality. Another stumbling block is the lack of a natural method of monotone approximation of an $(\omega, m)$-subharmonic function by smooth functions from this class. For plurisubharmonic functions, that is the case $m=n$,  this is possible (see e.g. \cite{blocki-kolodziej07, DP04}).
On K\"ahler manifolds Lu and Nguyen \cite{chinh-dong} employed the method of Berman \cite{berman} and  Eyssidieux-Guedj-Zeriahi \cite{egz13}
to construct smooth approximants of an $(\omega, m)$-subharmonic function. However this method requires the existence theorem for Hessian type equation, so it is far more complicated than the ones starting from convolutions with a smoothing kernel.
In the last section we carry out a similar construction to the one in \cite{chinh-dong} on Hermitian manifolds.

\bigskip

{\bf Acknowledgement.}   The research was partially supported by NCN grant \linebreak 2013/08/A/ST1/00312. A part of this work was done
while the first author visited E. Schr\"odinger Institute. He would like to thank the institution for hospitality and perfect working conditions. The second author is grateful to S\l awomir Dinew and Dongwei Gu for many useful discussions.


\section{Estimates  in $\bC^n$} 
\label{S1}

In this section we wish to develop tools, which correspond to results in  pluripotential theory, to study the Hessian equations with respect to a Hermitian form.
Some of those analogues, notably the Chern-Levine-Nirenberg inequalities, do not carry over trivially and they  require a careful examination of the  properties of  positive cones associated with elementary symmetric functions.
 The difficulty is to control the negative values of a vector belonging to such a cone. 
First we prove point-wise estimates for the cone in $\bR^n$ and then we express them in the language of differential forms which live in  the cone associated with a Hermitian metric $\omega$ in $\bC^n$. 
Next, we  use these results to prove basic "pluripotential"  estimates for $(\omega, m)$-subharmonic function such as the Chern-Levine-Nirenberg inequality, the Bedford-Taylor convergence theorem, the weak comparison principle and the like.
We refer to  \cite{garding59, ivochkina83, lin-trudinger94} and \cite{wang09} for  the properties of elementary symmetric functions which are used here.

\subsection{Properties of elementary positive cones} \label{s11}

Let $1 \leq m < n$ be two integers. We denote by
\[	 
	\Gamma_m = \{ \lambda = (\lambda_1, ..., \lambda_n) \in \bR^n : S_1(\lambda) > 0, ..., S_m(\lambda)>0\}
\]
the  symmetric positive cone associated with polynomials  \[ S_k(\lambda) = \sum_{1 \leq i_1 < \cdots < i_k \leq n} \lambda_{i_1} \lambda_{i_2} \cdots \lambda_{i_k}.\] We use the conventions
\begin{align*}
	S_0(\lambda) &= 1, \\
	S_k(\lambda) &=0 \quad \mbox{for } k>n \mbox{ or } k<0.
\end{align*}
For any fixed $t$-tuple $\{i_1, ..., i_t\} \subseteq \{1,...,n\}$, we write
\[
	S_{k; i_1 i_2...i_t} (\lambda) := S_k |_{\lambda_{i_1} =\cdots =\lambda_{i_t}=0}.
\]
So $S_{k; i_1 i_2...i_t}$ is the $k$-th order elementary symmetric function of $(n-t)$ variables $\{1, ...,n\}\setminus \{i_1, ..., i_t\}$. 
A property that we frequently use in the sequel is 
\begin{equation} \label{incre-s-m}
	S_m(\lambda) \leq S_m(\lambda + \mu) \quad \mbox{for every }
	\lambda, \mu \in \Gamma_m 
\end{equation}
(see \cite{garding59}). Furthermore, a characterisation of the cone $\Gamma_m$ (see e.g. \cite[Lemma 8] {ivochkina83}) tells  that if $\lambda \in \Gamma_m$, then 
\begin{equation} \label{hom-ineq}
	S_{k; i_1,...,i_t} (\lambda)>0
\end{equation}
for all $\{i_1, ...,i_t\} \subseteq \{1,...,n\}$, $k+t \leq m$. In particular, if $\lambda\in \Gamma_m$, then at least $m$ of the numbers $\lambda_1, ..., \lambda_n$ are positive. 
Hence, throughout this note we shall write the entries of $\lambda \in \Gamma_m$ in the  decreasing order
\begin{equation} \label{order-lam}
\lambda_1 \geq \cdots \geq \lambda_m \geq \cdots \lambda_p>0 \geq \lambda_{p+1} \cdots \geq \lambda_n
\end{equation}
(with  $p\geq m$ by the remark above).
It is clear that
\begin{align}
	S_k(\lambda) = S_{k; i} + \lambda_i S_{k-1;i} (\lambda). 
\end{align}
Therefore  we have the following expansion
\begin{equation} \label{expand-formula}
\begin{aligned}
	S_{k-1}(\lambda) 
&=	 S_{k-1;1} + \lambda_1 S_{k-2;1} \\
&=	S_{k-1;1}	 + \lambda_1 S_{k-2;12} + \lambda_1 \lambda_2 S_{k-3;12}\\
&=	S_{k-1;1} +  \lambda_1 S_{k-2;12} + \cdots + \lambda_1 \cdots \lambda_{k-2} S_{1;12\cdots (k-1)} + \lambda_1 \cdots \lambda_{k-1}.
\end{aligned}
\end{equation}
It follows from  \eqref{hom-ineq} that for $\lambda \in \Gamma_m$
\begin{equation} \label{lower-sk}
S_{m-1}(\lambda) \geq \lambda_1 \cdots \lambda_{m-1}.
\end{equation}
A more general statement is also true. 
\begin{lem} \label{s-m-1}
Let $1\leq k \leq m-1$ and $\{i_1, ...,i_{k}\} \subset \{1, ..., n\}$. Then, for every $\lambda \in \Gamma_m$,
\[
	|\lambda_{i_1} \cdots \lambda_{i_{k}}| \leq C_{n,k} S_{k}(\lambda),
\]
where $C_{n,k}$ depends only on $n,k$. 
\end{lem}

\begin{proof} 
Since $k\leq m-1$ and $\lambda \in \Gamma_m \subset \Gamma_{k+1}$, the expansion formula \eqref{expand-formula} gives that
\[
	S_{k} \geq \lambda_1 \cdots \lambda_{k}.
\]
Therefore, if $\{i_1, ...,i_{k}\} \subseteq \{1,...,p\}$, i.e. $\lambda_{i_t}>0$ for all $t=1,...,k$, then we are done by the arrangement \eqref{order-lam}. Otherwise, without loss of generality, we may assume that 
\[
	\lambda_{i_1} \geq \cdots \geq \lambda_{i_s} >0 > 
	\lambda_{i_{s+1}} \cdots \geq \lambda_{i_{k}}.
\]
For brevity we write 
\[
	A = \lambda_{i_1} \cdots \lambda_{i_{k}}.
\]
Consequently, 
\begin{align*}
	|A| 
&= 	(\lambda_{i_1} \cdots \lambda_{i_s}) |\lambda_{i_{s+1}} \cdots \lambda_{i_{k}}|  \\
&\leq 	(\lambda_{i_1} \cdots \lambda_{i_s}) |\lambda_{i_{k}}|^{k-s}.
\end{align*}
By \eqref{hom-ineq} we have that the  sum of any $n-k$ of entries $\lambda_i$ is positive and hence
\[
	|\lambda_{i_{k}}| \leq (p - k) \lambda_{k+1}.
\]
Note that $p \geq m\geq  k+1$. 
Thus,  it follows from the lower bound for $S_{k}$ that
\begin{align*}
	|A| 
&	\leq (p-k)^{k-s} \lambda_{i_1} \cdots \lambda_{i_s} (\lambda_{k+1})^{k-s} \\
&	\leq (n-k)^{k} \lambda_1 \cdots \lambda_{k} \\
&	\leq  (n-k)^{k} S_{k} (\lambda).
\end{align*}
Thus, the lemma is proven.
\end{proof}

We also get an upper bound for $S_m$ in terms of $S_{m-1;j}$ as follows. There exists $\theta=\theta(n,m)>0$ such that for any $j \leq m$,
\begin{equation} \label{lower-sk-lam}
	\lambda_j S_{m-1;j}(\lambda) \geq \theta S_m (\lambda) \quad
	\mbox{if } \lambda \in \Gamma_m.
\end{equation} 
Indeed, by 
\[
	S_m = S_{m;j} + \lambda_j S_{m-1;j}
\]
we see that \eqref{lower-sk-lam} is automatically true if $S_{m;j} \leq 0$. Otherwise, $S_{m;j}(\lambda)>0$, and we can estimate as follows:
\begin{align*}
	S_m 
&	\leq C_{n,m} \lambda_1 \cdots \lambda_m \\
&	\leq C_{n,m} \lambda_j S_{m-1;j},
\end{align*}
where the second inequality used \eqref{hom-ineq} and \eqref{expand-formula}. The inequality \eqref{lower-sk-lam} thus follows.


If $m=n$, then the following result is just a simple consequence of the Cauchy-Schwarz inequality.  

\begin{lem} \label{stablity-extra-lem}
Let $a = (a_1,...,a_n) \in \bR^n$ and $\lambda \in \Gamma_m$. Then,
\[
	\frac{n S_1(\lambda)}{S_m(\lambda)} \cdot \left(\sum_{i=1}^n |a_i|^2 S_{m-1;i}(\lambda) \right)
	\geq  \theta  \sum_{i=1}^n |a_i|^2,
\]
where $\theta= \theta(n,m)>0$ is the constant in \eqref{lower-sk-lam}.
\end{lem}

\begin{proof}
If $m=1$, then it is obvious. So we may assume that $m\geq 2$. Therefore, from  \eqref{order-lam}   and \eqref{lower-sk}     we have that
\[
	S_1 \geq \lambda_1, \quad S_{m-1,n} \geq S_{m-1;n-1} \geq \cdots \geq S_{m-1;1}>0.
\]
Moreover, by \eqref{lower-sk-lam}
\[
	\theta S_m \leq \lambda_1 S_{m-1;1}.
\]
Hence, for $m \geq 2$,
\[
	0< \frac{S_{m}}{S_{m-1;n}} \leq \cdots \leq \frac{S_m}{S_{m-1;1}} 
	\leq  \lambda_1/\theta \leq S_1/\theta ,
\]
and therefore
\[
\frac{n S_1}{ \theta S_m} \cdot \left(\sum_{i=1}^n |a_i|^2 S_{m-1;i} \right) 
\geq  \left(\sum_{i=1}^n \frac{1}{S_{m-1;i}}\right) \left(\sum_{i=1}^n |a_i|^2 S_{m-1;i} \right). 
\]
The lemma now follows by an application of the Cauchy-Schwarz inequality to the right hand side of the above inequality.
\end{proof}

\subsection{The positive cones associated with a Hermitian metric} \label{s12}

Let $\omega$ be a Hermitian metric on $\bC^n$ and let $\Omega$ be a bounded open set in $\bC^n$. Given a smooth Hermitian $(1,1)$-form $\gamma$ in $\Omega$, we say that $\gamma$ is $(\omega,m)$-positive if at any point $z\in \Omega$ it satisfies
\[
	\gamma^k \wedge \omega^{n-k}(z)  >0 \quad \mbox{for every} \quad k =1,...,m.
\]
Equivalently, in the normal coordinates with respect to $\omega$ at $z$,
diagonalizing  $\gamma = \ii \sum_i \lambda_i dz_i \wedge d\bar z_i$,  we have \[ \lambda = (\lambda_1,...,\lambda_n) \in \Gamma_m.\] 
This correspondence allows to express the estimates from Section~\ref{s11} in the language of differential forms. First of them can be found in \cite{blocki05}.
We denote the set of all $(\omega,m)$-positive smooth Hermitian $(1,1)$-forms by $ \Gamma_m(\omega, \Omega)$ or $\Gamma_m(\omega)$, when the domain $\Omega$ is clear from the context.


The inequality \eqref{incre-s-m} is equivalent to
\begin{equation} \label{incre-s-m-form}
	(\gamma + \eta)^m \wedge \omega^{n-m} 
	\geq \gamma^m\wedge \omega^{n-m} \quad \mbox{for every } 
	\gamma, \eta \in \Gamma_m(\omega).
\end{equation}

Lemma~\ref{s-m-1} gives a statement important for our applications.

\begin{lem} \label{form-s-m-1} 
Let $\gamma \in \Gamma_m(\omega)$ and $T$ is a smooth $(n-k,n-k)$-form with $1 \leq k \leq m-1$. Then,
\[
	|\gamma^{k} \wedge T/\omega^n| 
	\leq C_{n,k,\|T\|} \; \gamma^{k} \wedge\omega^{n-k}/\omega^n,
\]
where $C_{n,k,\|T\|}$ is a uniform constant depending only on $n,k$ and the sup norm of coefficients of $T$. 
\end{lem}

\begin{proof} 
Fix a point $P \in \Omega$. Choose a local coordinate system  at $P$ such that
\[
	\omega = \sum_{j =1}^n \ii dz_j \wedge d\bar z_j \quad \mbox{and}\quad
	\gamma= \sum_{j =1}^n \lambda_j \ii dz_j \wedge d\bar z_j.
\]
In those coordinates we write 
\[
T =	\sum_{|J|=|K| = n-k} T_{JK}  dz_J \wedge d \bar z_K. 
\]
In what follows, the computation is performed at $P$. We first have
\[
	\gamma^{k} 
	=k! \sum_{|I| = k, I \subseteq \{1,..,n\}} 
	\prod_{i_s \in I} \lambda_{i_s} dz_I \wedge d\bar z_I.
\]
The nonzero contribution in $\gamma^{k} \wedge T$
give only triplets of  multi-indices $I, J, K \subseteq \{1,...,n\}$ such that
\[
	I  \cup J = I  \cup K = \{1,...,n\},
\]
and $|I| = k$.
For such  sets $I, J, K$, we have
\[
	\frac{n!}{k!}(\ii)^{(n-k)^2} \gamma^{k} \wedge dz_J \wedge d\bar z_K/\omega^n=   \prod_{i_s \in I, |I| = k} \lambda_{i_s}.
\]
By Lemma~\ref{s-m-1} 
\begin{align*}
	  \prod_{i_s \in I,|I| = k} |\lambda_{i_s}| 
&\leq		 C_{n,k}   S_{k} (\lambda) \\
&=		C_{n,k} \binom{n}{k}	\gamma^{k} \wedge \omega^{n-k}/\omega^{n},
\end{align*}
where the constant $C_{n,k}$ depends only on $n,k$. 
Taking into account the coefficients $T_{JK}$, we get that each term in 
\[
	| \gamma^{k} \wedge T/\omega^n|
\] 
is bounded from above  by 
\[
 	 \gamma^{k} \wedge \omega^{n-k}/\omega^n,
\]
modulo a uniform constant $C_{n,k,\|T\|} = C_{n,k} \sup_{J,K}\|T_{JK}\|_\infty$, where $C_{n,k}$ may differ from the one above. Thus, the lemma follows.
\end{proof}

 We  need to generalise the last result  to the case of the wedge product of $k $ smooth Hermitian $(1,1)$-forms in $\Gamma_m(\omega)$ in place  of $\gamma ^k$. 
To do this, fix $k \leq m-2$ and  consider vectors $x = (x_1, ...,x_{k}) \in [0,1]^k \subset \bR^k$.
 The $\bR$-vector space of polynomials in $x$ of degree at most $k+1$ is denoted here by $P_{k+1}(\bR^k)$. Its dimension is equal to $d =\binom{2k+1}{k}$. 
We use  multi-indices $\alpha = (\alpha_1,...,\alpha_k) \in \bN^k$, with the length $|\alpha| := \alpha_1 + \cdots \alpha_k$, and ordered in some fixed fashion. The vector space $P_{k+1}(\bR^k)$ has the standard monomial basis
\[
	\{ x^\alpha = x_1^{\alpha_1}  \cdots x_k^{\alpha_k}:
	\quad |\alpha| \leq k+1 \} =: \{e_1, ..., e_d\},
\]
where $d = \binom{2k+1}{k}$. Choose 
a  set $X = \{X_1,...,X_d\}$, with  $X_i \in [0,1]^k$, 
such that the Vandermonde matrix 
\[
	V:= \{e_i(X_j)\}_{i,j =\overline{1,d}} 
\]
is non singular. 

Now,  for  $x \in [0,1]^{k}$ and $y = (\gamma_0,...,\gamma_k), \gamma_{j} \in \Gamma_m(\omega),$ consider the polynomial 
\begin{align*}
	P(x, y) & = (\gamma_0+ x_1 \gamma_1 + \cdots + x_k \gamma_k)^{k+1} \wedge T/\omega^n \\
&=:	\sum_{|\alpha| \leq k+1}  b_{\alpha}(y)x^\alpha,
\end{align*}
where $T$ is a smooth $(n-k-1,n-k-1)$-form and
\[ 
	b_{\alpha} (y) 
	= \frac{(k+1)!}{\alpha_1! \cdots \alpha_k!}\gamma_0^{k+1 - |\alpha|} \wedge \gamma_1^{\alpha_1} \wedge \cdots \wedge \gamma_k^{\alpha_k} \wedge T/\omega^n. \]
Put $
	\tau:= \gamma_0+ x_1 \gamma_1 + \cdots + x_k \gamma_k. 
$
By Lemma~\ref{form-s-m-1} we get that for every $x\in [0,1]^k$,
\[
|P(x,y)| \leq C \tau^{k+1} \wedge \omega^{n-k-1}/\omega^n 
\leq C (\gamma_0+ \cdots + \gamma_k)^{k+1} \wedge \omega^{n-k-1}/\omega^n.
\]
In particular  $|P(X_j , y)|$, for $X = \{X_1,...,X_d\}$ fixed above, are uniformly bounded by the right hand side of the last inequality.
The coefficients $b_{\alpha} (y) $ are computed by applying the inverse of $V$ to the column vector consisting of entries $P(X_j ,y)$.
Since $V$ is a fixed matrix we obtain the desired bound and the following statement.

\begin{cor} \label{poly-lem}
 Fix $k\leq m-2$.  Let $T$ be a smooth $(n-k-1,n-k-1)$-form. For $\gamma_0,...,\gamma_{k} \in \Gamma_m(\omega)$ we have
\[
	|\gamma_0 \wedge \cdots \wedge \gamma_{k}\wedge T/\omega^n| \leq
	C_{n,m,\|T\|} (\gamma_0 + \cdots + \gamma_{k})^{k+1} \wedge \omega^{n-k-1}/\omega^n,
\] 
where $C_{n,k,\|T\|}$ is a uniform constant depending only on $n,k$ and the sup norm of coefficients of $T$. 
\end{cor}

We end this subsection with the consequence of Lemma~\ref{stablity-extra-lem}. This will be used later in  the proof of the stability of solutions  to the Hessian equations.

\begin{lem} \label{stability-extra-2}
Let $\psi$ be a smooth function and $\gamma\in \Gamma_m(\omega)$. Then,
\[
	\frac{\ii \d \psi \wedge \bar\d \psi \wedge \gamma^{m-1}\wedge \omega^{n-m}}{\gamma^m \wedge \omega^{n-m}} 
	\cdot \frac{\gamma \wedge \omega^{n-1}}{\omega^n}
	 \geq  \frac{\theta\ii \d \psi \wedge \bar\d \psi \wedge \omega^{n-1}}{\omega^n},
\]
where $\theta = \theta(n,m)>0$.
\end{lem}

\begin{proof} 
It is an application of Lemma~\ref{stablity-extra-lem} in the normal coordinates with respect to $\omega$, where $a = (\psi_1, ...,\psi_n)$ with $\psi_i := \d \psi/\d z_i$ and $\lambda$ is the vector of  eigenvalues of $\gamma$ in those coordinates.
\end{proof}

\subsection{$(\omega,m)$-subharmonic functions}\label{s13}

Let $\Omega$ be a bounded open set in $\bC^n$. Assume that $\omega$ is a Hermitian metric on $\bC^n$. Fix an integer  $1\leq m < n$.

In this subsection we are going to define the notion of $(\omega,m)$-subharmonicity for non-smooth functions which is adapted from B\l ocki \cite{blocki05} and  Dinew-Ko\l odziej \cite{dinew-kolodziej12, dinew-kolodziej14}.
 We refer to papers by Lu \cite{chinh13}, Lu-Nguyen \cite{chinh-dong}, Dinew-Lu \cite{chinh-dinew} for more properties of this class of functions when $\omega$ is a K\"ahler metric.
Then, we will prove several  results which correspond to  basic pluripotential theory theorems from \cite{BT76, BT82}.
 

A $C^2(\Omega)$ real-valued function $u$ is called $(\omega,m)$-subharmonic if the associated form 
$ \omega_u := \omega + dd^c u$  belongs to $ \overline{\Gamma_m(\omega)}$.
It means that 
\[
	\omega_u^k \wedge \omega^{n-k} \geq 0 \quad \mbox{for every}\quad
	k =1,...,m.
\]

\begin{defn} \label{non-smooth-omega-m}
An  upper semi-continuous function $u: \Omega \to [-\infty, +\infty[ $ is called $(\omega,m)$-subharmonic if $u\in L^1_{loc}(\Omega)$ and for any collection of  $\gamma_1, ...,\gamma_{m-1} \in \Gamma_m(\omega)$
\[
	(\omega + dd^c u) \wedge \gamma_1 \wedge \cdots \wedge \gamma_{m-1} \wedge \omega^{n-m} \geq 0
\]
with the inequality understood in the sense of currents. 
\end{defn}
We denote  by $SH_{m}(\Omega, \omega)$ the set of all $(\omega,m)$-subharmonic functions in $\Omega$. We often write $SH_m(\omega)$ if the domain is clear from the context. 

\begin{observation} \label{equi-two-defn}
By results of G\aa rding \cite{garding59}, if $u \in C^2(\Omega)$, then $u$ is $(\omega,m)$-subharmonic  according to  Definition~\ref{non-smooth-omega-m} if and only if $\omega_u \in \overline{\Gamma_m(\omega)}$. In particular, we have that for $\gamma_1, ..., \gamma_k \in \Gamma_m(\omega)$, $k \leq m$, 
\[
	\gamma_1 \wedge \cdots \wedge \gamma_k \wedge \omega^{n-m} 
\]
is a strictly positive $(n-m + k, n-m+k)$-form.
\end{observation}
By \cite[Section~4, Eq. (4.8)]{michelsohn82} given $\gamma_1, ...,\gamma_{m-1} \in \Gamma_m(\omega)$ we can find a Hermitian metric $\tilde\omega$ such that 
\[
	\tilde \omega^{n-1} = \gamma_1 \wedge \cdots \wedge \gamma_{m-1}\wedge \omega^{n-m}.
\]
Thus, according to Definition~\ref{non-smooth-omega-m}, checking the $(\omega,m)$-subharmonicity of a given function $u$ can be reduced to verifying that $u$ is $(\tilde \omega,1)$-subharmonic for  a collection of Hermitian metrics $\tilde \omega$. 
Therefore, some properties of $(\omega,1)$-subharmonic functions are preserved by  $(\omega,m)$-subharmonic functions. Below we  list several of them and refer to \cite{dinew-kolodziej14} and \cite{chinh13} for more (if the K\"ahler condition does not play a role).

\begin{prop} \label{basic-facts} 
Let $\Omega$ be a bounded open set in $\bC^n$. 
\begin{itemize}
\item[{\bf (a)}]
If $u,v \in SH_m(\omega)$, then $\max\{u,v\} \in SH_m(\omega)$. 
\item[{\bf(b)}]
Let $\{u_\alpha\}_{\alpha \in I} \subset SH_m(\omega)$ be a family locally uniformly bounded from above,  and $u:= \sup_\alpha u_\alpha$. Then, the upper semicontinuous regularization $u^*$ is $(\omega,m)$-subharmonic. 
\end{itemize}
\end{prop}

It follows from Remark~\ref{equi-two-defn} (see also \cite{blocki05}) that  for any collection of $C^2(\Omega)$ $(\omega,m)$-subharmonic functions $u_1, ..., u_k$ with $1 \leq k \leq m$,
\begin{equation} \label{n-m-exponent}
	\omega_{u_1} \wedge \cdots \wedge \omega_{u_k} \wedge \omega^{n-m}
\end{equation}
is a positive form. 

The above properties of $(\omega, m)$-subharmonic functions are the same 
as in the  K\"ahler case. However, there are differences too.  If we replace the exponent $n-m$ by a smaller one, then the positivity of the differential form \eqref{n-m-exponent} is no longer true in general. 
This makes computations involving integration by parts more tricky.

Let $u_1, ..., u_p \in SH_m(\omega) \cap C^2(\Omega)$.
If we write
\[  \omega_{u_{j_1}} \wedge \cdots \wedge \hat\omega_{u_j} \wedge \cdots \wedge\omega_{u_{j_q}}, \]
where $j_1 \leq j \leq j_q$, the symbol hat indicates that the term does not appear in the wedge product. 
Then, we have 
\begin{equation} \label{dw}
\begin{aligned}
	d( \omega_{u_1} \wedge \cdots \wedge \omega_{u_p} \wedge \omega^{n-m}) &	= \sum_{j=1}^p d\omega \wedge \omega_{u_{1}} \wedge \cdots \wedge \hat\omega_{u_j} \wedge \cdots \wedge\omega_{u_{p}} \wedge \omega^{n-m} \\
&	\quad + (n-m) d\omega \wedge \omega_{u_{1}} \wedge \cdots \wedge \omega_{u_p} \wedge \omega^{n-m-1}; \\
\end{aligned}
\end{equation}
and
\begin{equation}\label{dd^cw}
\begin{aligned}
&	dd^c(\omega_{u_1} \wedge \cdots \wedge \omega_{u_p} \wedge \omega^{n-m}) \\
&	= \sum_{1\leq j \leq p} dd^c\omega \wedge  \omega_{u_{1}} \wedge \cdots \wedge \hat\omega_{u_j} \wedge \cdots \wedge\omega_{u_{p}}  \wedge \omega^{n-m} \\
&\quad + \sum_{i \neq j; 1\leq i,j \leq p} d\omega \wedge d^c\omega \wedge    \omega_{u_{1}} \wedge \cdots  \wedge \hat\omega_{u_i}  \wedge \cdots \wedge \hat\omega_{u_j} \wedge \cdots \wedge\omega_{u_{p}}\wedge \omega^{n-m}  \\
&\quad + 2(n-m)\sum_{1\leq j \leq p} d\omega\wedge d^c\omega \wedge   \omega_{u_{1}} \wedge \cdots \wedge \hat\omega_{u_j} \wedge \cdots \wedge\omega_{u_{p}} \wedge \omega^{n-m-1} \\
&\quad + (n-m) dd^c\omega \wedge \omega_{u_{1}} \wedge \cdots \wedge \omega_{u_{p}} \wedge \omega^{n-m-1} \\
&\quad + (n-m)(n-m-1) d\omega \wedge d^c \omega \wedge \omega_{u_{1}} \wedge \cdots \wedge \omega_{u_{p}} \wedge \omega^{n-m-2}.
\end{aligned}
\end{equation}
In those formulas  forms of  three types   appear:
\begin{align*}
	\omega_{u_1} \wedge \cdots \wedge \hat \omega_{u_j} \wedge \omega_{u_{p}} \wedge \omega^{n-m-1}, \\
	\omega_{u_1} \wedge \cdots \wedge \omega_{u_{p}} \wedge \omega^{n-m-1}, \\
	\omega_{u_1} \wedge \cdots \wedge \omega_{u_{p}} \wedge \omega^{n-m-2}.
\end{align*}
As $\omega_{u_i}$ is not a  positive $(1,1)$-form, these forms are not necessary positive (the exponent of $\omega$ is less than $n-m$). Therefore, in the estimates that follow, we can not apply directly the bounds for $dd^c \omega$ or $d\omega \wedge d^c \omega$ in terms of  $\omega^2$ or $\omega^3$   as in the  case of the Monge-Amp\`ere equation. 
Fortunately, the results from  previous subsections make the important estimates to go through  if $p \leq m-1$ (see Corollary~\ref{poly-lem}) 

We are ready to prove the Chern-Levine-Nirenberg (CLN) inequality which guarantees the compactness of a  sequence of Hessian measures provided that $(\omega,m)$-subharmonic potentials are uniformly bounded.

\begin{prop}[CLN inequality] \label{cln-ineq-local} Let $K \subset \subset U \subset\subset \Omega$, where $K$ is compact and $U$ is open. Let $u_1,..., u_k \in SH_m(\omega)\cap C^2(\Omega)$, $1\leq k \leq m$. Then, there exists a constant $C_{K, U,\omega}>0$ such that
\begin{align*}
	\int_K \omega_{u_1} \wedge \cdots \wedge \omega_{u_k} \wedge \omega^{n-k} 
\leq C_{K,U,\omega}\left(1+ \sum_{j=1}^k \|u_j\|_{L^\infty(U)}\right)^k.
\end{align*}
\end{prop}

\begin{proof} 
Observe that by \eqref{incre-s-m-form}
\[
	\omega_{u_1} \wedge \cdots \wedge \omega_{u_k} \wedge \omega^{n-k} \leq k^k \left(\omega + dd^c \frac{u_1+ \cdots + u_k}{k}\right)^k\wedge \omega^{n-k}.
\]
Set $u:= (u_1+\cdots + u_k)/k$. Thus we are reduced  to estimate
$
	\int_K \omega_u^k \wedge \omega^{n-k},
$ 
where $\omega_u \in \Gamma_m(\omega)$. 

We will prove it by induction in $k$. 
For $k=1$, let $\chi$ be a cut-off function such that $\chi =1$ on $K$ and $supp\, \chi \subset \subset U$. Then, 
\[
	\int_K\omega_u \wedge \omega^{n-1}
\leq  \int \chi \omega_u \wedge \omega^{n-1}  
= 	  \int \chi \omega^{n}  
	+ \int \chi dd^c u \wedge \omega^{n-1}.  
\]
It is clear that $\int \chi \omega^{n}   \leq C_{K,U,\omega}$ and by integration by parts we have
\[
	\int \chi dd^c u \wedge \omega^{n-1}
=	\int u dd^c (\chi \omega^{n-1}) 
\leq 	C_{K,U,\omega} \|u\|_{L^\infty(U)}.
\]
Thus, the CLN inequality holds for $k=1$. 
Suppose now that 
\[
	\int_K \omega_{u}^{k-1} 
\wedge  \omega^{n-k+1} \leq C_{K,U, \omega} (1+ \|u\|_{L^\infty(U)})^{k-1}.
\]
We need to infer the inequality
\[
	\int_K \omega_{u}^{k} \wedge  \omega^{n-k} 
	\leq C_{K,U, \omega} (1+ \|u\|_{L^\infty(U)})^{k}.
\]
Indeed, as
\begin{align*}
	\int_K \omega_{u}^{k} \wedge  \omega^{n-k}  
&\leq \int \chi \omega_{u}^{k} \wedge  \omega^{n-k}  \\
&=	\int \chi \omega_u^{k-1}\wedge \omega^{n-k+1} + 
	\int \chi dd^c u \wedge \omega_u^{k-1} \wedge \omega^{n-k},
\end{align*}
using the induction hypothesis it is enough to estimate the second term on the right hand side. The integration by parts gives 
\[
	\int \chi dd^c u \wedge \omega_u^{k-1} \wedge \omega^{n-k}
=	\int u dd^c (\omega_u^{k-1} \wedge \chi \omega^{n-k}).
\]
An elementary computation yields
\begin{align*}
	dd^c (\omega_u^{k-1} \wedge \chi \omega^{n-k}) 
&=	(k-1)(k-2) \omega_u^{k-3} \wedge d\omega \wedge d^c \omega \wedge \chi \omega^{n-k} \\
&\quad + (k-1) \omega_u^{k-2} \wedge dd^c \omega \wedge \chi \omega^{n-k} \\
&\quad - (k-1) \omega_u^{k-2} \wedge d^c\omega \wedge d (\chi \omega^{n-k}) \\
&\quad + (k-1) \omega_u^{k-2} \wedge d \omega \wedge d^c(\chi \omega^{n-k}) \\
&\quad + \omega_u^{k-1} \wedge dd^c (\chi \omega^{n-k}).
\end{align*}
Since $k\leq m$, applying  Lemma~\ref{form-s-m-1} for $\gamma= \omega_u$, we get that
\begin{align*}
&	\left| u dd^c (\omega_u^{k-1} \wedge \chi \omega^{n-k})\right| \\
&\leq 	C_{K,U, \omega} \|u\|_{L^\infty(U)} \left(\omega_u^{k-1} \wedge\omega^{n-k+1} + \omega_u^{k-2} \wedge \omega^{n-k+2} + \omega_u^{k-3} \wedge \omega^{n-k+3}\right).
\end{align*}
This implies that $\left| \int \chi dd^c u \wedge \omega_u^{k-1} \wedge \omega^{n-k} \right|$ is bounded by
\begin{align*}
& C_{K,U, \omega} \|u\|_{L^\infty(U)} \int_{supp\, \chi} \left(\omega_u^{k-1} \wedge\omega^{n-k+1} + \omega_u^{k-2} \wedge \omega^{n-k+2} + \omega_u^{k-3} \wedge \omega^{n-k+3}\right).
\end{align*}
Combined with the induction hypothesis this finishes the proof.
\end{proof}


For general $1<m<n$ and Hermitian metrics $\omega$, it is  not known yet that any $(\omega,m)$-subharmonic function is approximable by a decreasing sequence of smooth $(\omega,m)$-subharmonic functions. Therefore we need the following definition.

\begin{defn}[smoothly approximable functions] Let $u$ be an $(\omega,m)$-subharmonic function in $\Omega$.  We say that $u$ belongs to $\cA_{m}(\omega)$ if at each point $z \in \Omega$ there exists a ball $B(z, r) \subset\subset \Omega$, and smooth $(\omega,m)$-subharmonic functions $u_j$ in $B(z,r)$ decreasing to $u$ as $j$ goes to $\infty$. 
\end{defn}

Now, we shall develop "pluripotential theory" for $(\omega,m)$-subharmonic functions in the class $\cA_m(\omega)$.

\begin{prop}[wedge product] \label{wedge-prod} 
Fix a  ball $B(z_0,r)\subset \subset \Omega$.
Let $u_1,...,u_k \in SH_m(\omega) \cap C(\bar{B}(z_0,r))$, $1\leq k \leq m$. Assume that there exists a sequence of smooth $(\omega,m)$-subharmonic functions $u_1^j, ..., u_k^j$ decreasing to $u_1,...,u_k$ in $\bar B(z_0,r)$, respectively, then the sequence
\[
	(\omega +dd^c u_1^j) \wedge \dots \wedge (\omega+dd^c u_k^j)  \wedge \omega^{n-m} 
\]
converges weakly to a unique positive current, in $B(z_0,r)$, as $j$ goes to $+\infty$. 
\end{prop}

\begin{proof}
Thanks to Corollary~\ref{poly-lem} and the CLN inequality (Proposition~\ref{cln-ineq-local}), 
the proof is a standard modification of the Bedford and Taylor convergence theorem \cite{BT76, BT82}.
For notational simplicity we only give it in the case $k=m$, $u_1 = ...=u_{m}=u$ and $ u_1^j \equiv ... \equiv u_{m}^j \equiv u^j =:u_j$. The general case  follows by the same method.  Set $B:= B(z_0,r)$. Since $u$ is continuous on $\overline{B}$, it follows that $u_j \to u$ uniformly on that set. Hence, $\|u_j\|_{\infty}$ is uniformly bounded, where we denote here and below \[\|.\|_\infty := \sup_{\bar B}|.|.\] 
For any compact set $K \subset B$ we have
\[
	\int_K \omega_{u_j}^m \wedge \omega^{n-m} 
	\leq C_{K,B,\omega} (1+ \|u_j\|_\infty)^{m} 
\]
by the CLN inequality (Proposition~\ref{cln-ineq-local}). Therefore, the sequence 
\[
\omega_{u_j}^{m}  \wedge \omega^{n-m}, \quad j\geq 1, 
\]
is weakly compact in $B$. It implies that there exists a weak limit $\mu$ upon passing to a subsequence. 

It remains to check that every weak limit is equal to $\mu$. Suppose that $\{v_j\}_{j=1}^\infty$ and $\{w_j\}_{j=1}^\infty$ are two decreasing sequences of smooth $(\omega,m)$-subharmonic functions  converging to $u$. Since the statement is local we may assume that all functions are equal near the boundary of $B$ (see \cite{BT82, kolodziej05}). We need to show that for any test function $\chi \in C^\infty_c(B)$,
\[
	\left| \int_B \chi \omega_{v_j}^{m}  \wedge\omega^{n-m} - \int_B \chi \omega_{w_j}^{m}\wedge \omega^{n-m}\right| \longrightarrow 0
\]
as $j \to +\infty.$ Since $u$ is continuous on $\overline{B}$, it follows that both $\{v_j\}$ and $\{w_j\}$ converge uniformly to $u$ on that set. Hence, $\|v_j\|_\infty$, $\|w_j\|_\infty$ are uniformly bounded. By integration by parts we have
\[
	A_j:=\int_B \chi dd^c (v_j - w_j) \wedge T_j = \int_{B} (v_j - w_j) dd^c (\chi T_j) ,
\]
where $T_j = \sum_{s=0}^{m-1} \omega_{v_j}^s  \wedge \omega_{w_j}^{m-1-s} \wedge \omega^{n-m}$.  From Corollary~\ref{poly-lem} and  the above  proof of the CLN inequality we get that
\[
	A_j 	
	 \leq \|v_j - w_j\|_\infty  \int_{supp \; \chi} \|dd^c(\chi T_j)\|, 
\]	
where the last integral is controlled by \[C (1+ \|v_j\|_\infty)^{m-1} (1+ \|w_j\|_\infty)^{m-1}.\] Therefore,  we can conclude that  $ \lim_{j \to +\infty} A_j = 0$, and thus the result follows.
\end{proof}

\begin{cor} \label{we-pro-1}
Let $u_1,...,u_k \in \cA_m(\omega) \cap C(\Omega)$, $1\leq k \leq m$. Then, the wedge product
\[\omega_{u_1} \wedge \cdots \wedge \omega_{u_k} \wedge \omega^{n-m}\]
is a well-defined positive current of bidegree $(n-m+k, n-m+k)$. In particular, for $u\in \cA_m(\omega) \cap C(\Omega)$, the current
\[
	\omega_u^m \wedge \omega^{n-m}
\]
is the complex Hessian operator of $u$, which is a positive Radon measure in $\Omega$.
\end{cor}

\subsection{The comparison principle and maximality}

Let $\Omega$ be a bounded open set in $\bC^n$. 
Given $\omega$ a Hermitian metric there exists  a constant $B_\omega>0$, which we fix,  satisfying in $\bar{\Omega}$ 
\begin{equation} \label{tor-cst}
	-B_\omega \omega^2 \leq 2n dd^c \omega \leq B_\omega \omega^2, \quad
	-B_\omega \omega^3 \leq 4n^2 d\omega \wedge d^c \omega \leq B_\omega \omega^3.
\end{equation}

Thanks to Lemma~\ref{form-s-m-1} and Corollary~\ref{poly-lem}, the proof of \cite[Theorem 0.2]{KN1} can be adapted to Hessian operators and as a consequence we get the following domination principle.

\begin{prop}\label{domination-prin}
Let $u, v \in \cA_m(\omega) \cap C(\bar\Omega)$ be such that $u\geq v$ on $\d \Omega$. Assume that $(\omega + dd^c u)^m \wedge \omega^{n-m} \leq (\omega+ dd^cv)^m\wedge \omega^{n-m}$. Then $u\geq v$ on $\bar \Omega$.
\end{prop}

\begin{proof} See \cite[Corollary 3.4]{KN1}. We remark here that if $u, v$ belong to $C^2(\Omega)$, then the corollary can be proven simply by using the ellipticity of the Hessian operator \cite[Lemma B]{CNS85}.
\end{proof}

The above proposition shows that if $u \in \cA_m(\omega) \cap C(\bar \Omega)$ and $\omega_u^m \wedge \omega^{n-m} =0$, then it is maximal in $\cA_m(\omega) \cap C(\bar \Omega)$. We shall see that a stronger result is true. 
First, we recall a couple of facts from classical potential  theory. For a general fixed Hermitian metric $\gamma$ in $\bC^n$ and  a Borel set $E \subset \Omega$ we define
\[
	C_{\gamma}(E) = \sup\left\{\int_E dd^c w \wedge \gamma^{n-1} : w \mbox{ is } \gamma-\mbox{subharmonic in }\Omega, \;0\leq w \leq 1\right\}.
\]
\begin{prop} \label{quasi-con}
Every $\gamma-$subharmonic function $u$ is quasi-continuous with respect to the capacity $C_{\gamma}$, i.e. for any $\vepsilon>0$, there exists an open set $U \subset \Omega$ such that $C_{\gamma}(U)<\vepsilon$ and $u$ restricted to $\Omega\setminus U$ is continuous.
\end{prop}

\begin{lem} \label{app-gamma}Every $\gamma-$subharmonic in a neighbourhood of the closure of $\Omega$ is the limit of a decreasing sequence of smooth $\gamma-$subharmonic functions, in $\Omega$. \end{lem}

Next, we strengthen the domination principle. It is usually applied locally, so we formulate it for $\Omega$ being a ball.

\begin{thm}[maximality] \label{maximality}
Let $\Omega$ denote a ball and let $v \in SH_m(\omega) \cap L^\infty(\Omega)$. Let $u \in \cA_m(\omega)\cap C(\bar \Omega)$ be the uniform limit of $\{u_j\}_{j=1}^\infty \subset SH_m(\omega) \cap C^\infty(\bar \Omega)$. Suppose that $G:=\{u <v\}\subset \subset \Omega$.  If $\omega_u^m \wedge \omega^{n-m} =0$ on $G$, then $G$ is empty.
\end{thm}

To prove the theorem, we need the following result.

\begin{lem} \label{special-cp}  Fix $0< \vepsilon <1$ and the constant $B_\omega$ in \eqref{tor-cst}. Let $v \in SH_m(\omega) \cap L^\infty(\Omega)$, with $\Omega$ denoting a ball. Assume $\{u_j\}_{j=1}^\infty \subset SH_m(\omega) \cap C^\infty(\bar \Omega)$  converges uniformly to $u$ as $j \to \infty$ in $\bar \Omega$.  Denote $S (\vepsilon ):= \inf_{\Omega} [u -(1-\vepsilon)v]$ and $U(\vepsilon, t) := \{u < (1-\vepsilon)v + S(\vepsilon ) +t \}$ for $t>0$. Suppose that $U(\vepsilon, t_0) \subset \subset \Omega$ for some $t_0>0$. Then, for $0< t < \min\{\vepsilon^3/16B_\omega, t_0\}$
\[
	\vepsilon \int_{U(\vepsilon,t)}  \omega_u^{m-1} \wedge \omega^{n-m+1}
	\leq (1 + \frac{Ct}{\vepsilon^m}) \int_{U(\vepsilon,t)} \omega_u^{m} \wedge \omega^{n-m},
\]
where $C$ is a uniform constant depending only on $n,m, B_\omega$.
\end{lem}

\begin{proof} 
By Corollary~\ref{we-pro-1}, it is enough to show that 
\[
\vepsilon \int_{U_j(\vepsilon,t)}\omega_{u_j}^{m-1} \wedge \omega^{n-m+1}\leq 	 (1 + \frac{Ct}{\vepsilon^m})\int_{U_j(\vepsilon,t)} \omega_{u_j}^{m} \wedge \omega^{n-m} ,
\]
where $U_j(\vepsilon, t)$ is the sublevel set corresponding to $u_j$ and $v$ defined as above. In other words, we only need to prove the lemma under the assumption that  $u$ is smooth and strictly $(\omega,m)$-subharmonic, i.e. $\omega_u \in \Gamma_m(\omega)$ (achieved by  considering the sequence $(1-1/j)u_j$, $j\geq 1$). 

Moreover, since $\vepsilon \omega_{u}^{m-1} \wedge \omega^{n-m+1} \leq  \omega_{(1-\vepsilon)v}\wedge \omega_{u}^{m-1} \wedge \omega^{n-m}$, it suffices to prove that
\begin{equation} \label{ml2-eq1}
	\int_{U(\vepsilon,t)} \omega_{(1-\vepsilon)v}\wedge \omega_{u}^{m-1} \wedge \omega^{n-m}
	\leq (1 + \frac{Ct}{\vepsilon^m})\int_{U(\vepsilon,t)} \omega_{u}^{m} \wedge \omega^{n-m}.
\end{equation}
Since $\omega_u^{m-1} \wedge \omega^{n-m}>0$, applying \cite[ Eq.~(4.8)]{michelsohn82} we can write 
\begin{equation} \label{michelsohn}
\gamma^{n-1} := \omega_u^{m-1} \wedge \omega^{n-m}
\end{equation}
for some  Hermitian metric $\gamma$. By the definition of an     $(\omega,m)$-subharmonic function, 
\[
\omega_v\wedge \gamma^{n-1} \geq 0.
\]
Solving the linear elliptic equation  we can write $\omega \wedge \gamma^{n-1} = dd^c w \wedge \gamma^{n-1}$ for some smooth $\gamma-$subharmonic function $w$ in $\Omega$. Therefore, if we set $\tilde v := v + w$, then $\tilde v$ is a $\gamma$-subharmonic function.
Having this property we can use the proof of \cite[Proposition 3.1]{BT76} and the quasi-continuity  of $\tilde v$ (equivalently that  of $v$), from Proposition~\ref{quasi-con} to get that
\[
	\int_{U(\vepsilon,t)} dd^c (1-\vepsilon) v \wedge \gamma^{n-1}
	\leq \int_{U(\vepsilon,t)} dd^c u \wedge \gamma^{n-1} +
	\int_{U(\vepsilon,t)} [(1-\vepsilon)v + S_\vepsilon + t -u] dd^c \gamma^{n-1}.
\]
It implies that
\begin{equation} \label{repeat-gamma}
	\int_{U(\vepsilon,t)} \omega_{(1-\vepsilon)v} \wedge \gamma^{n-1}
	\leq \int_{U(\vepsilon,t)} \omega_u \wedge \gamma^{n-1} 
	+ t \int_{U(\vepsilon,t)} \|dd^c \gamma^{n-1}\|,
\end{equation}
where $\|dd^c \gamma^{n-1}\|$ is the total variation of $dd^c \gamma^{n-1}$. Furthermore,  we can use Lemma~\ref{form-s-m-1} to bound $\|dd^c \gamma^{n-1}\|$ from above by
\[
R := C(\omega_u^{m-1}\wedge \omega^{n-m+1} + \omega_u^{m-2}\wedge \omega^{n-m+2} + \omega_u^{m-3}\wedge \omega^{n-m+3}),
\]
where $C$ depends only on $X, \omega$, $n,m$. 
Therefore, the inequality \eqref{ml2-eq1} will follow if we have that 
\[
	\int_{U(\vepsilon,t)} R
	\leq \frac{C}{\vepsilon^m} \int_{U(\vepsilon,t)} \omega_u^m \wedge \omega^{n-m}
\]
for every $0< t < \min\{\vepsilon^3/16B_\omega, t_0\}$. Writing 
$
	a_k := \int_{U(\vepsilon,t)} \omega_u^k \wedge \omega^{n-k}, 
$ for  $0 \leq k \leq m$, we need  to show that 
\[
	a_k \leq \frac{C a_m}{\vepsilon^m} .
\]
As in \cite[Theorem~2.3]{KN1} we shall verify that for $0<t< \delta:= \min\{\vepsilon^3/16B_\omega, t_0\}$,
\[
	\vepsilon a_k \leq a_{k+1} + \delta B_\omega (a_k + a_{k-1} + a_{k-2}),
\]
where we understand $a_k \equiv 0$ if $k<0$. Indeed, since $u$ is smooth and strictly $(\omega,m)$-subharmonic, the inequality \eqref{repeat-gamma} applied for $\gamma _k^{n-1} : = \omega_u^k\wedge \omega^{n-k-1}>0$, $0\leq k \leq m-2$ (see \eqref{michelsohn}), gives that
\[
	\int_{U(\vepsilon, t)} \omega_{(1-\vepsilon)v}  \wedge \gamma _k^{n-1}
	\leq \int_{U(\vepsilon, t)} \omega_u \wedge \gamma _k^{n-1}
	+ t \int_{U(\vepsilon,t)} \|dd^c \gamma _k^{n-1}\|.
\]
By \eqref{dd^cw}, \eqref{tor-cst} and Lemma~\ref{form-s-m-1} we have
\[
	\int_{U(\vepsilon,t)} \|dd^c \gamma _k^{n-1}\|
\leq	 B_\omega (a_k + a_{k-1} + a_{k-2}).
\]
Moreover, since $v$ is a bounded $(\omega,m)$-subharmonic function, one also has
\[
	\vepsilon \int_{U(\vepsilon, t)} \omega \wedge \gamma _k^{n-1}
	\leq \int_{U(\vepsilon, t)} \omega_{(1-\vepsilon)v}  \wedge \gamma _k^{n-1} .
\]
Combining last three  inequalities we get that for $0< t < \delta$,
\[
	\vepsilon a_k \leq a_{k+1} + \delta B_\omega (a_k + a_{k-1} + a_{k-2}). 
\]
Thus  the proof of the lemma  follows.
\end{proof}


\begin{proof}[Proof of Theorem~\ref{maximality}] 
Suppose that $\{u<v\}$ is not empty, then for $\vepsilon >0$ small enough,  we  have $\{u < (1-\vepsilon)v + \inf_{\Omega}[w - (1-\vepsilon)v] + t \} \subset \{u<v\}$ for any $0<t \leq t_0$, 
where $t_0>0$ depends on $u, v, \vepsilon$. Applying Lemma~\ref{special-cp}  we have for $0< t \leq \min\{\vepsilon^{m+3}/16B_\omega, t_0\}$ 
\begin{align*}
	\vepsilon \int_{U(\vepsilon,t)} \omega_u^{m-1}\wedge \omega^{n-m+1} 
	\leq 	C \int_{U(\vepsilon,t)} \omega_u^{m} \wedge \omega^{n-m} =0,
\end{align*}
where $C$ is independent of $t$.
Therefore, $\omega_u^{m-1} \wedge \omega^{n-m+1}=0$ in $U(\vepsilon,t)$ for $0<t \leq t_1 $, where $t_1:= \min\{\vepsilon^{m+3}/16B_\omega, t_0\}$.  
Thus we can iterate this argument  to  get that $\omega_u^{m-2} \wedge \omega^{n-m+2} = ... =\omega^n =0$ in $U(\vepsilon,t_1)$. This is impossible and the  proof of the theorem follows.
\end{proof}

\begin{observation} \label{maximality-cp}
The  statement of Theorem~\ref{maximality} holds true if we replace $\bar \Omega$ by a compact Hermitian manifold, with the same proof modulo obvious modifications.
\end{observation}

We end this subsection by proving a volume-capacity inequality which corresponds to the one in  \cite{dinew-kolodziej14}. This inequality was the key ingredient to study local integrability of $m-$subharmonic functions.

\begin{defn}[capacity] \label{m-cap} For any Borel set $E \subset \Omega$,
\[
	cap_{m,\omega}(E)
	:= \sup\left\{\int_E (\omega + dd^c v)^m \wedge \omega^{n-m} : v\in \cA_m(\omega) \cap C(\Omega), \; 0\leq v\leq 1\right\}.
\]
\end{defn}

\begin{lem}[local volume-capacity inequality]  \label{vol-cap-local}
Let $1<\tau< n/(n-m)$. There exists a constant $C = C(\tau)$ such that for any Borel set $E \subset \Omega$,
\[
	V_\omega(E) \leq C [cap_{m,\omega}(E)]^{\tau},
\]
where $V_\omega(E) := \int_E \omega^n$.
\end{lem}

The exponent here is optimal because if we take $\omega = dd^c |z|^2$, then the explicit formula for $cap_m(B(0,r))$ in $\Omega = B(0, 1)$ with $0< r<1$, provides an example.

\begin{proof} From \cite[Proposition~2.1]{dinew-kolodziej14} we know that $V_\omega(E) \leq C [cap_m(E)]^{\tau}$  with 
\[
	cap_m(E) = \sup\{ \int_E (dd^c w)^m \wedge \omega^{n-m}: w\in \cA_m \cap C(\Omega), 0\leq w \leq 1\},
\]
which is the capacity related to  $m-\omega$-subharmonic functions in $\Omega$ and the class $\cA_{m}$ consists of all $m-\omega$-subharmonic functions which are locally approximable by a decreasing sequence of smooth $m-\omega$-subharmonic functions in $\Omega$.
 Note that the argument in \cite{dinew-kolodziej14} remains valid for  non-K\"ahler $\omega$  since the mixed form type inequality used there still holds by stability estimates for the Monge-Amp\`ere equation.

Therefore, the proof will follow if we can show that $cap_{m}(E)$ is less than $cap_{m,\omega}(E)$. Since $\omega$ is globally defined there exists a  constant $C>0$ such that
\[
	\frac{1}{C} dd^c \rho \leq \omega \leq C dd^c \rho,
\]
where $\rho = |z|^2 -A\leq 0$. We can choose $C$ such that $|\rho/C| \leq 1/2$. Take $0 \leq w \leq 1/2$ a continuous $m-\omega$-subharmonic in $\cA_{m}$, then it is easy to see that
\[
	\int_E (dd^cw)^m \wedge \omega^{n-m}  
	\leq \int_E \left(\omega + dd^c(w - \frac{\rho}{C})\right)^m \wedge \omega^{n-m} \leq cap_{m,\omega}(E).
\]
Hence, $cap_m(E) \leq 2^n  cap_{m,\omega}(E)$. 
\end{proof}

\section{Hessian equations on compact Hermitian manifolds}

In this section we study Hessian equations on  a compact $n$-dimensional Hermitian manifold  $(X,\omega)$. To do this we need first to transfer the local results from the previous section
to the manifold setting. Then we apply them to prove results on the existence and stability of solutions of Hessian equations.
Finally, we prove that every $(\omega,m)$-subharmonic function can be  approximated  by a decreasing sequence of smooth $(\omega, m)$-subharmonic function on $X$. 
This allows to replace assumptions on $\cA_m(\omega)$ by just $SH_m(\omega)$ in statements. In what follows we  use our notations as in \cite{KN1, KN2}, we write $L^1(\omega^n)$ for $L^1(X, \omega^n)$, $\|.\|_p:= \|.\|_{L^p(X, \omega^n)}$ and $\|.\|_\infty:= \sup_X |.|$.

\subsection{Pluripotential estimates for $(\omega, m)$-subharmonic functions}
 Fix an integer $1 \leq m < n$.
By  means of partition of unity we carry over the local construction from Section~\ref{S1} onto the compact Hermitian manifold $X$.

\begin{defn} An upper semi-conitnuous function $u: X \to [-\infty, +\infty[$ 
is called $(\omega,m)$-subharmonic in $X$ if $u \in L^1(\omega^n)$ and   $u \in SH_m(U, \omega)$ for each coordinate patch $U\subset\subset X$. \end{defn}

We denote by $SH_m(X, \omega)$ or $SH_m(\omega)$ the set of all $(\omega,m)$-subharmonic functions  in $X$. 
Similarly, we say that $u \in \cA_m(\omega)$ if $u\in SH_m(\omega)$ and there exists a decreasing sequence of smooth $(\omega,m)$-subharmonic functions on $X$  which converges to $u$ (globally). 
So, if $u \in \cA_m(\omega)$, then for any coordinate patch $U \subset \subset X$ we have $u \in \cA_m(U, \omega)$. 
Thus the properties of  $\cA_m(U, \omega)$ (e.g. Proposition~\ref{basic-facts}, Hessian measures, the Bedford-Taylor convergence theorem, etc.)   are also valid for $\cA_m(\omega)$.

Below we state several results  which are analogues of those from \cite{DK12}. We omit the proofs which are similar and  require only the local properties.

\begin{prop}[CLN inequalities] \label{cln-ine} 
Let $\varphi_1,...,\vphi_m \in \cA_m(\omega)\cap C(X)$ and $0\leq \vphi_1,...,\vphi_m \leq 1$. Then there exists a uniform constant $C>0$ such that
\[
	\int_X \omega_{\vphi_1} \wedge \cdots \wedge \omega_{\vphi_m} \wedge\omega^{n-m} \leq C.
\]
\end{prop}
 
The following lemma seems to be classical (see e.g.  H\"ormander's book \cite{H94}).  
\begin{lem} \label{l1-compactness}
Let $\vphi \in SH_m(\omega)$ with $\sup_X \vphi =0$. There exists a uniform constant $C=C(X,\omega)>0$ such that\[  \int_X |\vphi| \omega^n \leq C.  \]
Consequently, the family $\{\vphi \in SH_m(\omega): \sup_X \vphi =0\}$ is compact in $SH_m(\omega)$ with respect to $L^1(\omega^n)-$topology, i.e. for any sequence $\vphi_j \in SH_m(\omega)$ with $\sup_X \vphi_j=0$, $j\geq 1$, there exists a subsequence $\{\vphi_{j_k}\}$ such that $\vphi_{j_k}$ converges to $\vphi \in SH_m(\omega)$ as $j_k \to + \infty$ in $L^1(\omega^n)$.
\end{lem}

\begin{proof}
The first part is from \cite[Section~2, p.8]{TW13b}, where the proof used only the fact that $\vphi$ is a smooth $(\omega,1)$-subharmnic function, i.e. 
\[  n dd^c \vphi \wedge \omega^{n-1}/\omega^n \geq -n,  \] coupled with the existence of Green function for the Gauduchon metric in the conformal class of $\omega$. Since every $(\omega,1)$-subharmonic function is approximated by decreasing sequence of smooth $(\omega,1)$-subharmonic functions, so we get the statement for general $(\omega, m)$- subharmonic functions. 
The second part follows from Proposition~\ref{basic-facts} and requires only properties of $(\omega,1)$-subharmonic functions.   \end{proof}

The estimates of the decay of volume of sublevel sets follow directly from Lemma~\ref{l1-compactness}. 
We use the notation
\[
	V_\omega(E) := \int_E \omega^n.
\]

\begin{cor} \label{decay-vol-sub-sets}
Let $\vphi \in SH_m(\omega)$ with $\sup_X \vphi =0$. Then, for any $t>0$,
\[
	V_{\omega}(\{\vphi<-t\}) \leq C/t,
\]
where $C>0$ is a uniform constant.
\end{cor}

Following \cite{BT82} and  \cite{kolodziej03} we define the capacity related to  the Hessian equations.

\begin{defn}[capacity]\label{cap-cpt} For a Borel set $E \subset X$ 
\[
	cap_{m,\omega}(E) := 
	\sup\{\int_E \omega_\rho^m \wedge \omega^{n-m} 
	: \rho \in \cA_m(\omega) \cap C(X), \; 0\leq \rho \leq 1\}.\]
\end{defn}

Then,  as in the local case, we have the estimate with the  sharp exponent. 

\begin{prop} \label{vol-cap-2} 
Fix $1< \tau < n/(n-m)$. There exists a uniform constant $C = C(\tau, X, \omega)>0$ such that for any Borel set $E \subset X$, 
\[
	V_\omega(E) \leq C [cap_{m,\omega}(E)]^{\tau}.
\]
\end{prop}

\begin{proof} 
The basic idea is from \cite{dinew-kolodziej14}. Surprisingly, it is enough
to  use the estimates for the Monge-Amp\`ere equation to obtain a sharp  bound related to capacity defined in terms of more general Hessian equations.
One could infer the statement from the local counterpart, but due to the difficulties with approximation by smooth $(\omega, m)$-subharmonic functions that approach would be
more technical than a direct proof   (like  \cite{chinh13} in the K\"ahler case). This requires the estimates in the Hermitian setting \cite{KN1}.

Without loss of generality we assume that $V_\omega(E)>0$. Denote by ${\bf 1}_E$ the characteristic function of $E$. By \cite[Theorem~0.1]{KN1} we can find a continuous 
$\omega$-plurisubharmonic function $u$ on $X$ with $\sup_X u =0$ and a constant $b>0$ solving 
\[
	\omega_u^n = b \; {\bf 1}_E \omega^n.
\]
Set $p = \frac{m\tau}{n(\tau -1)} >1$. We will need the lower bound for $L^p$-norm of $b \,{\bf 1}_E$.

{\bf Fact.} There exists  a uniform constant $c_0>0$ depending on $X, \omega, p$ such that 
\[
	\|b \, {\bf 1}_E\|_p  \geq c_0.
\]
Indeed, suppose that it were not true, then there would be a sequence of Borel sets $\{E_j\}_{j=1}^\infty$ that
\[
	1 \geq \|b_j \,{\bf 1}_{E_j}\|_p \searrow 0 \quad \mbox{as } j \to 
	+\infty. 
\]
By \cite{KN1, KN2} we know that  for $0< t\leq t_{min}$ ($t_{min}>0$ depending only on $X, \omega$)
\[
	t^n \hbar(t) \leq  C \|b_j \,{\bf 1}_{E_j} \|_1 \leq C \|b_j \,{\bf 1}_{E_j}\|_p \searrow 0,
\]
where the function $\hbar (t)$ is the inverse function of $\kappa(t)$ defined  in \cite[Theorem~5.3]{KN1}. This leads to a contradiction for a fixed $t=t_{min}$.

Thus, by a priori estimates for Monge-Amp\`ere equations \cite[Corollary~5.6]{KN1} we have 
\begin{equation}\label{vc-eq1}
	\|u\|_\infty \leq C \|b \, {\bf 1}_E\|_p^\frac{1}{n} 
	= C b^{1/n} [V_\omega(E)]^{1/pn}.
\end{equation}
We observe that by the proof of \cite[Proposition~1.5]{cuong15} for $-1 \leq w \leq 0$
\[
	\int_X \omega_w^n \geq \int_X \omega^n - C\|w\|_\infty, 
\]
where $C = C(X, \omega)$. Hence, 
there exists $0< \delta = \delta(X, \omega)<1$ such that if $\|u\|_\infty \leq \delta$ then 
$
	\int_X \omega_u^n \geq V_\omega(X)/2,
$  i.e. $b \geq V_\omega(X)/2 V_\omega(E).$
Now we consider two cases. 

{\bf Case 1:} If $\|u\|_\infty > \delta$, then, by \eqref{vc-eq1} 
\begin{equation} \label{vc-eq2}
	\|u\|_\infty + 1 \leq  (C+ C/\delta)\,  b^{1/n} [V_\omega(E)]^{1/pn}.
\end{equation}
The mixed form type inequality \cite[Lemma 1.9]{cuong15} gives $\omega_u^m\wedge \omega^{n-m} \geq b^{m/n} {\bf 1}_E$. Hence,  by definition of capacity we have 
\begin{align*}
	cap_{m,\omega}(E) 
&	\geq	\frac{1}{(1+\|u\|_\infty)^m} \int_E (\omega + dd^c u)^m \wedge \omega^{n-m} \\
&	\geq \frac{1}{(1+\|u\|_\infty)^m} \int_E b^{m/n} {\bf 1}_E \omega^n \\
&	\geq	\frac{b^{m/n} V_\omega(E)}{C_1 b^{m/n} [V_\omega(E)]^{m/pn}} \\
&	= \frac{[V_\omega(E)]^{1-m/pn}}{C_1},
\end{align*}
where we used \eqref{vc-eq2} for the last inequality and $C_1 = (C + C/\delta)^m$. Therefore, we have
\[ V_\omega(E) \leq C [cap_{m,\omega}(E)]^{1 + m/(pn-m)}.\] Plugging the value of $p= \frac{m \tau}{n(\tau -1)}$  gives the desired inequality. 

{\bf Case 2:} If $\|u\|_\infty \leq \delta <1$, then $b \geq V_\omega(X)/2 V_\omega(E)$. Again, by definition we have
\begin{align*}
	cap_{m,\omega}(E) 
&	\geq \int_E \omega_u^m \wedge \omega^{n-m} \\
&	\geq \int_E b^\frac{m}{n} {\bf 1}_E \omega^n \\
&	\geq \left(\frac{V_\omega(X)}{2 V_\omega(E)}\right)^\frac{m}{n} 
\cdot V_\omega(E).
\end{align*}
It implies that $V_\omega(E) \leq C [cap_{m, \omega}(E)]^{n/(n-m)}.$ 
Thus we complete the proof.
\end{proof}

Let us recall that, by the definition,  the constant $B>0$ satisfies on $X$
\begin{equation} \label{torsion-global}
	-B \omega^2 \leq 2n dd^c \omega \leq B \omega^2, \quad
	-B \omega^3 \leq 4n^2 d\omega \wedge d^c \omega \leq B\omega^3.
\end{equation}

For general Hermitian metric $\omega$ the Hessian measures do not preserve the volume of manifold, so the classical comparison principle \cite{BT82, kolodziej05} is no longer true (see \cite{DK12}). However, a weaker form will be enough for several applications as it is proven in \cite{KN1, KN2}. We state below the analogue for Hessian operators.

\begin{thm}[weak comparison principle] \label{weak-cp}
Let $\vphi, \psi \in \cA_m(\omega)\cap C(X)$. Fix $0<\vepsilon<1$ and use the following notation $S(\vepsilon):= \inf_X [\vphi- (1-\vepsilon)\psi]$ and  $U(\vepsilon, s) := \{\vphi<(1-\vepsilon)\psi + S(\vepsilon)+s\}$ for $s>0$. Then, for $0<s< \vepsilon^3/16B$,
\[
	\int_{U(\vepsilon,s)} \omega_{(1-\vepsilon)\psi}^m \wedge \omega^{n-m}
	\leq (1+ \frac{C s}{\vepsilon^m})\int_{U(\vepsilon,s)} \omega_\vphi^m \wedge \omega^{n-m} ,
\]
where $C>0$ is a uniform constant depending only on $n,m,\omega$.
\end{thm}

\begin{proof}
It follows from the argument in \cite[Theorem 0.2]{KN1} with the aid of Corollary \ref{poly-lem}.
\end{proof}

Thanks to the weak comparison principle we can estimate the rate of the decay of capacity of sublevel sets not far from the minimum point.

\begin{lem}  \label{capacity-growth-sublevel-sets} 
Fix $0<\vepsilon<3/4$ and $\vepsilon_B := \frac{1}{3} \min\{\vepsilon^m, \frac{\vepsilon^3}{16B}\}$. Consider $\vphi, \psi \in \cA_m(\omega) \cap C(X)$ with $\vphi \leq 0$ and $-1 \leq \psi \leq 0$. 
With  $U(\vepsilon,s)$  defined as in the previous theorem, 
for any $0<s,t<\vepsilon_B$, we have
\begin{equation} \label{cgss-eq}
	t^m cap_{m,\omega}(U(\vepsilon,s)) 
	\leq C \int_{U(\vepsilon, s+t)} \omega_\vphi^m\wedge \omega^{n-m},
\end{equation}
where $C>0$ depends only on $X,\omega$.
\end{lem}

\begin{proof}
See the arguments in  \cite[Lemma~5.4, Remark~5.5]{KN1} by using the above weak comparison principle (Theorem~\ref{weak-cp}).
\end{proof}

The preparations above were needed for the proof of a priori estimates for solutions to Hessian equations with the right hand side in $L^p , \ p>n/m.$ 
We follow the method from \cite{kolodziej98, kolodziej03}  with small variations.

\begin{lem} \label{vol-sub-level-property}  Under assumptions and notations of Lemma~\ref{capacity-growth-sublevel-sets}. Assume furthermore that 
\[
	\omega_\vphi^m \wedge \omega^{n-m} =f \omega^n
\]
for $f \in L^p(\omega^n)$, $p>n/m$. Fix  $0< \alpha < \frac{p -\frac{n}{m}}{p(n-m)}$. Then, there exists a constant $C_\alpha = C(\alpha, \omega)$ such that for any $0<s,t<\vepsilon_B$, 
\[
	t \left[V_\omega( U(\vepsilon, s))\right]^\frac{1}{m \tau} \leq
	C_\alpha \|f\|_p^\frac{1}{m} \left[ V_\omega(U(\vepsilon, s+t)) \right]^\frac{1 + m \alpha}{m \tau},
\]
where $\tau = \frac{(1+m \alpha)p}{p-1}<n/(n-m)$.
\end{lem}

\begin{proof}
It is elementary that  
\begin{equation}\label{tau} 
	0< \alpha < \frac{p -\frac{n}{m}}{p(n-m)} \Leftrightarrow
	\frac{p}{p-1}< \tau=  \frac{(1+m \alpha)p}{p-1}<\frac{n}{n-m}.
\end{equation}
By the volume-capacity inequality (Proposition~\ref{vol-cap-2}) and Lemma~\ref{capacity-growth-sublevel-sets} we have 
\[
	t^m \left[V_\omega(U(\vepsilon,s))\right]^\frac{1}{\tau}
\leq  C_\alpha \;t^m \; cap_{m,\omega}(U(\vepsilon,s)) 
\leq C_\alpha \cdot C \int_{U(\vepsilon, s+t)} f \omega^n.
\]
The H\"older inequality implies that
\[
	t^m \left[V_\omega(U(\vepsilon,s))\right]^\frac{1}{\tau}
\leq C_\alpha  \|f\|_p \left[V_\omega(U(\vepsilon, s+t))\right]^\frac{p-1}{p}.
\]
Taking  $m-$th root of  both sides   and plugging the value of $\tau$ we get the desired inequality.
\end{proof}

Thanks to this lemma we get a uniform estimate for the solution of Hessian equations with $L^p , \ p>n/m$ control of  the right hand side.

\begin{thm} \label{a-priori-estimate} Fix $0<\vepsilon<3/4$ and $\vepsilon_B := \frac{1}{3} \min\{\vepsilon^m, \frac{\vepsilon^3}{16B}\}$. Let  $\vphi, \psi \in \cA_m(\omega) \cap C(X)$ satisfy $-1\leq \psi \leq 0$ and $\vphi \leq 0$. Assume that \[\omega_\vphi^m \wedge \omega^{n-m} = f \omega^n\] with $f \in L^p(\omega^n), p>n/m$. Put \[U(\vepsilon, s) = \{\vphi<(1-\vepsilon) \psi + \inf_X [\vphi -(1-\vepsilon) \psi] +s \},\] 
and fix $0< \alpha < \frac{p -\frac{n}{m}}{p(n-m)}$. Then, there exists a contant $C_\alpha = C(\alpha, \omega)$ such that for $0<s<\vepsilon_B$, 
\[
	s\leq 4 C_\alpha \|f\|_p^\frac{1}{m} 
	\left[V_\omega(U(\vepsilon,s)) \right]^\frac{\alpha}{\tau},
\]
where $\tau= \frac{(1+m \alpha)p}{p-1}$.
\end{thm}

\begin{proof} 
First, for $0< \alpha < \frac{p -\frac{n}{m}}{p(n-m)}$ we define 
\[
	a(s):= [V_\omega(U(\vepsilon, s))]^\frac{1}{m\tau}, \quad 
	C := C_\alpha \|f\|_p^\frac{1}{m}.
\] 
It follows from Lemma~\ref{vol-sub-level-property} that for any $0< s,t < \vepsilon_B$,
\begin{equation} \label{ape-eq1}
	t a(s) \leq C \left[a(s+t)\right]^{1+m\alpha}.
\end{equation}
The function $a(x)$ satisfies
\begin{equation} \label{ape-eq2}
	\lim_{x \to s^-} a(x) = a(s) \quad \mbox{and}
	\lim_{x \to s^+} a(x)=:a(s^+) \geq a(s).
\end{equation}
To finish the proof, we shall show that for any $0< s < \vepsilon_B$
\[
	s \leq \frac{2^{1+ m \alpha}}{2^{m \alpha} -1} \cdot  C [a(s)]^{m \alpha}.
\]
The argument  is  similar to the proof of \cite[Theorem~5.3]{KN1}, however  here it  is simpler, so we include the proof for the sake of completeness.

Fix $s_0:=s \in (0, \vepsilon_B)$. Let us define by induction the sequence $s_i, i \geq 1$ as follows.
\begin{equation} \label{ape-eq4}
	s_i := \sup \{0 \leq x \leq s_{i-1} : a(s_{i-1}) \geq 2 a(x)\}.
\end{equation}
Since $a(0) =0$ and $a(x)>0$ for $x>0$, it follows from the first equality in \eqref{ape-eq2} that 
\[
	s_0 > s_1 > \cdots > s_i \searrow 0 \quad \mbox{as } i \to +\infty. 
\]
(If $a(0^+)>0$, then $s_N = s_{N+1} = \cdots =0$ for some $1\leq N < +\infty$.) By \eqref{ape-eq2} and the definition \eqref{ape-eq4} we get that
\[
	2 a(s_{i}) \leq a(s_{i-1}) \leq 2 a(s_i^+).
\]
Hence, by \eqref{ape-eq1},
\[
	s_{i-1} - s_i = \lim_{x \to s_i^+} (s_{i-1} -x) \leq C [a(s_{i-1})]^{1+m\alpha}/a(s_i^+).
\]
It follows that 
\begin{align*}
	s_{i-1} - s_i \leq 2C [a(s_{i-1})]^{m\alpha} 
&	\leq 2C (1/2^{m\alpha}) [a(s_{i-2})]^{m \alpha}  \\
&	\leq \cdots \leq \\
&	\leq 2C (1/2^{m\alpha})^{i-1} [a(s_0)]^{m \alpha}.
\end{align*}
Thus,
\begin{align*}
	s = \sum_{i =1}^\infty (s_{i-1} - s_i) 
&	\leq 2^{1+ m\alpha} C\sum_{i=1}^\infty (1/2^{m\alpha})^i [a(s_0)]^{m \alpha} \\
&	=	\frac{2^{1+ m\alpha}C}{2^{m \alpha} -1} [a(s)]^{m \alpha} .
\end{align*}
This completes the proof.
\end{proof}

From the statement of Theorem~\ref{a-priori-estimate}, we can derive the uniform estimate by taking $\vepsilon =1/2$ and $\psi =0$ and combining it  with the estimate of the decay of volume of sublevel set
 (Corollary~\ref{decay-vol-sub-sets}). Thus we get that if $\omega_\vphi^m\wedge \omega^{n-m} = f\omega^n$ with $0\leq f \in L^p(\omega^n)$, $p>n/m$ and $\vphi$ is normalized by 
$\sup_X \vphi = -1$, then for any $0< s < \vepsilon_B$
\[
	s \leq \frac{C_\alpha \|f\|_p^\frac{1}{m}}{ \left|- \inf_X \vphi - s\right|^\frac{(p-1)\alpha}{p (1+ m \alpha)}},
\]
where $0< \alpha < \frac{p -\frac{n}{m}}{p(n-m)}$ is fixed. It  leads to 
\begin{equation} \label{uniform-est}
	\|\vphi\|_\infty \leq C \|f\|_p^{\frac{1}{m} \cdot \frac{p (1+ m \alpha)}{(p-1) \alpha}},
\end{equation}
where $C = C(\alpha, p, \omega, X)$. Note that here we have used the fact that there exists a uniform lower bound for $\|f\|_p$ similar to the one in \cite{KN1, KN2}. Though this case is simpler. Indeed, it follows from Theorem~\ref{a-priori-estimate} that for $s = \vepsilon_B/2$, 
\[
	\|f\|_p^\frac{1}{m} \geq	\frac{\vepsilon_B}{8 C_\alpha [V_\omega(X)]^\frac{\alpha}{\tau}}.
\]
This gives an explicit bound.

\subsection{Existence of weak solutions and stability}

The existence of weak solutions to the Monge-Amp\`ere equations on compact Hermitian manifold has been obtained recently in \cite{KN1} where the technique is quite different from \cite{kolodziej05}. We will adapt those techniques to the Hessian equation.

Let us start with a quantitative version of \cite[Corollary~5.10]{KN1} (see also \cite[Theorem~3.1]{dinew-kolodziej14} for the similar result in the K\"ahler case).

\begin{thm} \label{stab-l-r}
Let $u, v \in \cA_m(\omega) \cap C(X)$ be such that $\sup_X u =0$ and $v\leq 0 $. Suppose that $\omega_u^m \wedge \omega^{n-m} = f \omega^n$, where $f\in L^p(\omega^n), p>n/m$. Fix $0< \alpha < \frac{p -\frac{n}{m}}{p(n-m)}$. Then,
\[
	\sup_X(v - u) \leq C \|(v-u)_+\|_1^{1/ ap^*},
\]
where the constant $a = 1/p^* + m(m+2)+(m+2)/\alpha$, and  $C$ depends only on $\alpha, p, \omega, \|f\|_p$ and $\|v\|_\infty$.
\end{thm}

\begin{proof} 
By the uniform estimate \eqref{uniform-est} $\|u\|_\infty$ is controlled by $\|f\|_p$. After a rescaling we may assume that  $\|u\|_\infty, \|v\|_\infty \leq 1$. We wish to estimate
 $-S:= \sup_X (v - u) > 0$ in terms of $\|(v - u)_+\|_1$ as in the K\"ahler case \cite{kolodziej03}. Suppose that 
\begin{equation} \label{slr-eq2} \|(v - u)_+\|_1 \leq \vepsilon^{a p^*} 
\end{equation}
for $0< \vepsilon << 3/4$ and $a>0$ (to be determined later).
Let  
\[ \hbar(s) := (s/4C_\alpha \|f\|_p^\frac{1}{m})^\frac{1}{\alpha} \]  be the inverse function of $4C_\alpha \|f\|_p^\frac{1}{m}  s^\alpha$ in Theorem~\ref{a-priori-estimate}.  Consider sublevel sets $U(\vepsilon, t) = \{u< (1-\vepsilon) v + S_\vepsilon +t \}$, where $S_\vepsilon = \inf_X [u -(1-\vepsilon)v]$. It is clear that \[ S - \vepsilon \leq S_\vepsilon \leq S.\] 
Therefore, $U(\vepsilon,2t) \subset \{u < v + S+ \vepsilon +2t\}$. Then, $(v - u)_+ \geq |S| - \vepsilon -2t>0$ for $0< t < \vepsilon_B$ and $0< \vepsilon < |S|/2$ on the latter set (if $|S| \leq 2 \vepsilon$ then we are done).

By Lemma~\ref{capacity-growth-sublevel-sets} and the H\"older inequality, we have
\begin{align*}
	cap_{m,\omega}(U(\vepsilon,t)) 
	\leq \frac{C}{t^m} \int_{U(\vepsilon,2t)} f \omega^n 
&	\leq \frac{C}{t^m} \int_X \frac{(v -u)_+^{1/p^*}}{(|S| - \vepsilon -2t)^{1/p^*}}
		f \omega^m\\
&	\leq \frac{C \|f\|_p}{t^m (|S| - \vepsilon -2t)^{1/p^*}} \|(v - u)_+\|_1^{1/p^*}.
\end{align*}
Moreover, by Theorem~\ref{a-priori-estimate} \[\hbar(t) \leq[ V_\omega(U(\vepsilon,t))]^\frac{1}{\tau} \leq  C \; cap_{m,\omega}(U(\vepsilon,t)),\] where $\tau = (1+m\alpha) p^*$ and $C$ also depends  on $\alpha$. Combining these inequalites, we obtain
\[
	(|S| - \vepsilon - 2t)^{1/p^*} 
	\leq \frac{C \|f\|_p}{t^m \hbar(t)} \|(v - u)_+\|_1^{1/p^*}.
\]
Therefore, using \eqref{slr-eq2},
\begin{align*}
|S| 
&	\leq \vepsilon + 2t 
	+ \left(\frac{C \|f\|_p}{t^m \hbar(t)} \right)^{p^*}\|(v - u)_+\|_1 \\
&	\leq 3 \vepsilon + \left(\frac{C \|f\|_p \vepsilon^a}{t^m \hbar(t)} \right)^{p^*} .
\end{align*}
Recall that $\vepsilon_B = \frac{1}{3} \min\{\vepsilon^m, \frac{\vepsilon^3}{16B}\}$. So, taking 
$
	t = \vepsilon_B/2 \geq \vepsilon^{m+2} $
we have
\[
	\hbar(t) = \left(\frac{t}{4C_\alpha \|f\|_p^\frac{1}{m}}\right)^{1/\alpha} 
	\geq \frac{C \vepsilon^{(m+2)/\alpha}}{\|f\|_p^\frac{1}{m\alpha}}.
\]
If we choose $a = 1/p^* + m(m+2)+(m+2)/\alpha$, then 
\[\left(\vepsilon^{a}/\vepsilon^{m(m+2)+ (m+2)/\alpha}\right)^{p^*} =\vepsilon.\] 
Hence  $|S| \leq C \vepsilon$ with $C = C(\alpha, p, \omega, \|f\|_p)$. Thus,
\[
	\sup_X(v -u) \leq C \|(v - u)_+\|_1^{1/ap^*}.
\]
This is  the stability estimate we wished to show. 
\end{proof}

Applying the above theorem twice we get the symmetric (with respect to $u$ and $v$) form of this result.

\begin{cor} \label{stabilit1} Fix $\alpha>0$ and $a>0$ as in Theorem~\ref{stab-l-r}. Suppose that $u, v \in \cA_m(\omega) \cap C(X)$, normalized $\sup_X u = \sup_X v =0$, satisfy 
\[
	\omega_u^m \wedge \omega^{n-m}= f \omega^n, \quad
	\omega_v^m \wedge \omega^{n-m}= g \omega^n,
\]
where $0 \leq f, g \in L^p(\omega^n)$, $p>n/m$. Then, 
\[
	\|u - v\|_\infty \leq C \|u -v\|_1^{1/ap^*},
\]
where $C = C(\alpha, p, \|f\|_p, \|g\|_p, X, \omega)>0$.
\end{cor}

On compact non-K\"ahler manifolds  we can only expect to solve the Hessian equation up to multiplicative constant on the right hand side.
One needs to know that those constants stay bounded as long as  the given functions on the right hand side are bounded in $L^p$.

\begin{lem} \label{constant-bound} Suppose that $u\in SH_m(\omega) \cap C^\infty(X)$ satisfies
\[
	\omega_u^m \wedge \omega^{n-m} = c\,f \omega^n, 
\]
where $f \in L^p(\omega^n)$, $p>n/m$, and $\int_X f \omega^n>0$. Then, 
\[
	 c_{min} \leq c \leq 1/c_{min}
\]
for a uniform constant $c_{min} = C(\|f\|_p, \|f^{1/m}\|_1, X, \omega)>0$.
\end{lem}

\begin{proof} It is a consequence of mixed form type inequality and the a priori estimate in Theorem~\ref{a-priori-estimate}. The proof is similar as for  the Monge-Amp\`ere equation  \cite[Lemma~5.9]{KN1}.
\end{proof}

Thanks to the work of Sz\'ekelyhidi \cite{szekelyhidi15} and Zhang \cite{dzhang15},  the Hessian equation has a  smooth solution when the right hand side is smooth and positive. Using approximation procedure as in \cite{KN1}  and the stability (Corollary~\ref{stabilit1}) we get the following existence result. 
Note that the solution is obtained as a uniform limit of a sequence of smooth functions, therefore it automatically  belongs to $\cA_m(\omega)$.

\begin{thm}[existence]\label{existence-lp}
Let  $0 \leq f \in L^p(\omega^n), p>n/m$ satisfy $\int_X f \omega^n >0$. There exists $u\in \cA_m(\omega) \cap C(X)$ and a constant $c>0$ satisfying
\[
	(\omega+ dd^c u)^m\wedge \omega^{n-m} = c f \omega^n.
\]
\end{thm}

\begin{observation} \label{const-unique}
As in \cite{cuong15}, it follows from the weak comparison principle (Theorem~\ref{weak-cp}) that the constant $c>0$ is uniquely defined by $f$.
\end{observation}

By adapting the method in \cite{KN2} we get the following stability statement for the Hessian equation on compact Hermitian manifolds.

\begin{prop} \label{stability2}
Suppose that $u, v \in SH_m(\omega) \cap C^\infty(X)$, $\sup_X u=\sup_X v =0$, satisfy
\[
	\omega_u^m \wedge \omega^{n-m} = f \omega^n, \quad
	\omega_v^m \wedge \omega^{n-m} = g \omega^n, \quad
\]
where $f, g\in L^p(\omega^n)$, $p>n/m$. Assume that \[ f \geq c_0>0 \] for some constant $c_0$. Fix $0< a < \frac{1}{m+1}$. Then,
\[
	\|u-v\|_\infty \leq C \|f-g\|_p^{a}
\]
where the constant $C$ depends on $c_0, a, p, \|f\|_p, \|g\|_p, \omega, X$.
\end{prop}

\begin{proof} The proof follows the one in \cite[Theorem~3.1]{KN2}  with the  difference  that we need here the smoothness assumption on $u,v$  in order to use the mixed form type inequality \cite{garding59}. This inequality is likely to be true in general setting (see \cite{kolodziej05, cuong15}), but at the moment we do not have it.
In Section~\ref{S1} we have provided estimates for elementary symmetric functions which are needed to make the arguments in \cite{KN2} go through. 
We only point out where those arguments  should be modified. 

Note that now both $f$ and $u$ are smooth. Use the notation
 \[ \vphi := u-v \quad \mbox{and} \quad  T = \sum_{k=0}^{m-1} \omega_u^k \wedge \omega_v^{m-1-k} \wedge \omega^{n-m}.\]
By Corollary~\ref{poly-lem} we still have for a continuous function $w \geq 0$ on X and a Borel set $E\subset X$, that
\[
	\left| \int_E w dd^c T \right| \leq C \|w\|_{L^\infty(E)} (1+ \|u\|_\infty)^m(1+ \|v\|_\infty)^m.
\]
So the inequality \cite[eq. (3.16)]{KN2} is valid. Next,  the inequality corresponding to the one in the proof of  \cite[Lemma 3.6]{KN2}  has the following form:
\[
	\frac{\omega_u \wedge \omega^{n-1}}{\omega^n} \cdot 
	\frac{\ii\d\vphi \wedge \bar\d \vphi \wedge\omega_u^{m-1} \wedge \omega^{n-m}}{\omega^n}
\geq	\frac{\omega_u^m \wedge \omega^{n-m}}{\omega^n} \cdot 
	\frac{\theta \ii \d \vphi \wedge \bar\d \vphi \wedge \omega^{n-1}}{\omega^n},
\]
where $\omega_u \in \Gamma_m$.
 This is exactly the content of Lemma~\ref{stability-extra-2} applied for $\gamma = \omega_u$ and $\vphi$. There is  an extra constant $\theta>0$  here,
but it causes no harm as it only depends  on $n,m$. 
\end{proof}

\subsection{Approximation $(\omega,m)$-subharmonic functions} 
We are going to show the approximation property for $(\omega,m)$-subharmonic functions on $X$ for every $1< m<n$. The case $m=1$ is classical. The case $m=n$, i.e. for  quasi-plurisubharmonic functions,  is a  result due to Demailly (see \cite{blocki-kolodziej07} for a simple proof). When $\omega$ is K\"ahler the approximation property for $(\omega,m)$-subharmonic functions has been recently proven  by Lu and Nguyen \cite{chinh-dong}. They use the viscosity solutions  and ideas from \cite{berman} and \cite{egz13}. 
By a  similar approach, but without reference to viscosity solutions,  we generalise the approximation theorem in \cite{chinh-dong} to the case of general Hermitian metric $\omega$. 

The following theorem is essentially  contained in the work of Sz\'ekelyhidi \cite{szekelyhidi15}.

\begin{thm} \label{exist-smooth-hes-type}
Let $H$ be a smooth function on $X$. Then,  there exists a unique $u \in SH_m(\omega) \cap C^\infty(X)$ solving the Hessian equation 
\[
	(\omega + dd^c u)^m \wedge \omega^{n-m} = e^{u+H}\omega^n.
\] 
\end{thm}

\begin{proof}
The uniform estimate follows from the maximum principle. We claim that 
there exists a constant $C = C(H, \omega)$ such that
\[
	\|u\|_\infty \leq C.
\]
Indeed, suppose that $u$ attains maximum at $x \in X$. Then, $dd^c u(x) \leq 0$. Hence, at $x$,
\[
	e^{u(x) + H(x)} = (\omega + dd^c u)^m \wedge \omega^{n-m}/\omega^n \leq \omega^n/\omega^n=1.
\]
It implies that $e^{\sup_X u} \leq e^{- \inf_X H}$. 
Similarly, $e^{\inf_X u} \geq e^{- \sup_X H}$.

\begin{lem}[the Hou-Ma-Wu Laplacian estimate] \label{lap-est}  We have
\[
	\sup_{X} |\ddbar u| \leq C (1+ \sup_{X} |\nabla u|^2),
\]
where the constant $C$ depends on $\|u\|_\infty, \omega, H$.
\end{lem}

\begin{proof} We follow  the proof in \cite{szekelyhidi15} which generalised the result of Hou-Ma-Wu \cite{hou-ma-wu} to Hermitian manifolds. We only need to adjust our notation to the one  in \cite{szekelyhidi15}. Write
\[
	\omega = \ii \sum \omega_{j \bar k} dz_j \wedge d\bar z_k.
\] 
Let  $(\omega^{j \bar k})$ be the inverse matrix of $(\omega_{j \bar k})$ and consider \[ A^{ij} = \omega^{j \bar p} (\omega_{i \bar p} + u_{i \bar p}) =: \omega^{j \bar p} g_{i \bar p}.\]
Then, the equation is equivalent to 
\[
	F(A) = u + H,
\]
where 
\[
	F(A) = \log S_{m} (\lambda([A^{ij}])),
\]
with  $S_{m}$ denoting the elementary symmetric polynomial of degree $m$.
Without loss of generality we may assume that $z_0$ is the origin $0$ and the coordinates $z$ are chosen  as in \cite[Section 4]{szekelyhidi15}. 

From now on we use the notation and the computations  in \cite[Section~4]{szekelyhidi15} with $\alpha \equiv \chi \equiv \omega$. Since $\|u\|_\infty \leq C$, where $C$ is a uniform constant  and $\omega$ is a positive form, then $\underline u \equiv 0$ is the subsolution in the sense used in \cite{szekelyhidi15}.  
When the right hand side is independent of $u$  the proof is given in \cite{szekelyhidi15}. 
A small modification is required for the present case. As the equation is now
 \[
F(A) = u + H,
\] 
 the computations will change accordingly at each step. We  need to use the differentiation  at $0$ to get
\begin{align*}
	 u_p + H_p = F^{kk} g_{k \bar k p},  \\
	u_{1 \bar 1} + H_{1 \bar 1} = F^{pq, rs} g_{p\bar q 1} g_{r \bar s \bar 1} 
	+ F^{kk} g_{k \bar k 1 \bar 1}.
\end{align*}
Since $\cF = \sum F^{kk} > \tau$ and $u_{1\bar 1}$ is controlled by $\lambda_1>1$, the second equation above is enough to get the inequality $(81)$ in   \cite{szekelyhidi15}:
\[
	F^{kk} \tilde \lambda_{1,k\bar k} \geq - F^{pq, rs}g_{p \bar q 1} g_{r\bar s \bar 1}
	- 2F^{kk} Re(g_{k \bar 1 1} \overline{T^1_{k1}}) - C_0 \lambda_1 \cF. 
\]
Again, if we replace $h_p$ there by $u_p + H_p$, the inequality $(95)$ in   \cite{szekelyhidi15}  holds true:
\[
	F^{kk} u_{pk\bar k} u_{\bar p} \geq - C_0 K \cF - \epsilon_1 F^{kk} \lambda_k^2 - C_{\epsilon_1} \cF K.
\]
The rest of the proof is unchanged. So we get the lemma.
\end{proof}

Thus, we have proven the Hou-Ma-Wu type second order estimate which enables us to use the blow-up argument, due to Dinew and Ko\l odziej \cite{dinew-kolodziej12}, to get the gradient estimate 
(see also its variations by Tosatti-Weinkove \cite{TW13b} and by 
Sz\'ekelyhidi  \cite{szekelyhidi15}). 
Consequently, we also get a priori estimates for $|\d\bar\d u|$. Then, $C^{2, \alpha}$ estimates follows from the Evans-Krylov theorem, see e.g. \cite{twwy14}. By bootstrapping arguments we get $C^\infty$  estimates for the equation. 

Finally, the existence follows by the standard continuity method through the family
\[
	\log (\omega_{u_t}^m \wedge \omega^{n-m}/\omega^n) = u_t + t H
\] 
for $t \in [0,1]$. The uniqueness is a simple consequence of the maximum principle.
\end{proof}

We also need the existence and uniqueness of weak solutions of the Hessian type equation. We refer to \cite{cuong15} for more details about weak solutions to this equation in the case $m=n$.

\begin{thm} \label{stability-hes-type} Let $0 \leq f \in L^{p}(\omega^n)$, $p>n/m$ be such that $\int_X f \omega^n >0$. Assume that 
$\{f_j\}_{j\geq 1}$ are smooth and positive functions on $X$ converging in $L^p(\omega^n)$ to $f$ as $j \to +\infty$. Assume that
$u_j\in SH_m(\omega)\cap C^\infty(X)$ solves
\begin{equation} \label{ehtlp-eq1}
	\omega_{u_j}^m \wedge \omega^{n-m} = e^{u_j} f_j\omega^n.
\end{equation}
Then, $u_j$ converges uniformly to $u\in \cA_m(\omega)\cap C(X)$ as $j \to +\infty$, which is the unique solution in $\cA_m(\omega)\cap C(X)$ of 
\begin{equation} \label{ehtlp-eq2}
	\omega_u^m \wedge \omega^{n-m} = e^u f \omega^n .
\end{equation}
\end{thm}

\begin{proof} Set $M_j := \sup_X u_j$. Using the argument \cite[Claim~2.6]{cuong15} we get that $M_j$ are uniformly bounded. Set 
$
	\tilde u_j := u_j - M_j.
$ The equation \eqref{ehtlp-eq1} reads
\[
	\omega_{\tilde u_j}^m \wedge \omega^{n-m} = e^{\tilde u_j + M_j} f_j \omega^n.
\]
Then, $\{\tilde u_j\}_{j\geq 1}$ is relatively compact in $L^1(\omega^n)$ (Lemma~\ref{l1-compactness}). Passing to a subsequence, still writing $\tilde u_j$, we obtain a Cauchy sequence in $L^1(\omega^n)$. By Corollary~\ref{stabilit1} it follows that $\{\tilde u_j\}_{j\geq}$ is a Cauchy sequence in $C(X)$. Therefore, it converges uniformly to a solution $\tilde u \in \cA_m(\omega)$ of $\omega_{\tilde u}^m \wedge \omega^{n-m} = e^{\tilde u+M} f \omega$, where  $M= \lim_j M$. Rewriting $u = \tilde u+M$ we get that $u_j$ converges uniformly to $u$ which satisfies $\omega_u^m \wedge \omega^{n-m} = e^u f \omega^n$. 

By the weak comparison principle (Theorem~\ref{weak-cp}) the equation \eqref{ehtlp-eq2} has at most one solution in $\cA_m(\omega)\cap C(X)$ (see e.g. \cite[Lemma~2.3]{cuong15}). Thanks to  this, we conclude that the sequence $u_j$ converges uniformly to the unique solution $u$ because every convergent subsequence in $L^1(\omega^n)$ does.
\end{proof} 

We are ready to prove the main result of this subsection.

\begin{lem}[approximation property]  \label{appro-global}
For any $u\in SH_m(X, \omega)$ there exists a decreasing sequence of smooth $(\omega,m)$-subharmonic functions on $X$ converging to $u$ point-wise. In particular $SH_m(X, \omega) \equiv \cA_m(X, \omega)$.
\end{lem}

\begin{proof}
The general scheme is borrowed from Berman \cite{berman}, Eyssidieux-Guedj-Zeriahi \cite{egz13} (used also in \cite{chinh-dong}). However, to make the argument work we have to employ results which allow to extend the proof from the K\"ahler context to the  Hermitian one. 

Take $u$  an $(\omega,m)$-sh function. As $\max\{u, -j\} \in SH_m(\omega)$ for any $j\geq 1$, without loss of generality we may assume that $u$ is bounded. Suppose that $u \leq h \in C^\infty(X)$, where the function $h$ may not belong to $SH_m(\omega)$. Consider the largest $(\omega, m)$-sh function $\tilde h$ which is smaller or equal  than $h$.  The function $\tilde h$ can be obtained by taking upper semicontinuous regularization of \[\sup\{v \in SH_m(\omega) \cap L^\infty(X): v \leq h\}.\] Then, it is clear that $\tilde h$ is a $(\omega, m)$-sh and $u \leq \tilde h \leq h$. We are going to show that $\tilde h$ can be approximated by a decreasing sequence of smooth $(\omega, m)$-subharmonic functions, i.e. $\tilde h\in \cA_m(\omega)$. Once this is done, we also obtain $u\in \cA_m(\omega)$ by leting $h \searrow u$ and choosing an appropriate sequence of approximants of $\tilde h \searrow u$. 

Since $h\in C^\infty(X)$, we can write $\omega_h^m \wedge \omega^{n-m} = F \omega^n$ with $F$ being a  smooth function on $X$. We take the non-negative part $F_* = \max\{F, 0\}$, and then  a smooth approximation of it to obtain non-negative and smooth function $\tilde F \geq F_*$.  Using the existence of a  smooth $(\omega,m)$-solution to the complex Hessian type equation (Theorem~\ref{exist-smooth-hes-type}), we get for $0< \vepsilon \leq 1$,
\[
	\omega_{\tilde w_\vepsilon}^m \wedge \omega^{n-m} = e^{\frac{1}{\vepsilon}(\tilde w_\vepsilon - h)} [\tilde F + \vepsilon] \omega^n, 
\]
where $\tilde w_\vepsilon \in SH_m(\omega) \cap C^\infty(X)$. 

It is easy to see, by maximum principle, that $\tilde w_\vepsilon \leq h$ and $\tilde w_\vepsilon$ is decreasing in $\vepsilon$. That means $\tilde w_\vepsilon \nearrow$ as $\vepsilon \searrow 0$ and is bounded from above by $h$. Taking limits on both sides as $\tilde F \to F_*$ uniformly, by Theorem~\ref{stability-hes-type}  we get (for any fixed $\vepsilon$)
that $\tilde w_\vepsilon \to w_\vepsilon \in \cA_m(\omega) \cap C(X)$ uniformly and $w_\vepsilon $ is also increasing as $\vepsilon \searrow 0$. Moreover, at the limit we have
\[
	\omega_{w_\vepsilon}^m \wedge \omega^{n-m} = e^{\frac{1}{\vepsilon}(w_\vepsilon - h)} (F_* + \vepsilon) \omega^n. 
\]
Since $w_\vepsilon \leq h$, the right hand side is uniformly bounded in $L^\infty(X)$. The monotone sequence of continuous $(\omega,m)$-subharmonic functions $\{w_\vepsilon\}_{\vepsilon>0}$ is bounded by $h$, therefore it is Cauchy in $L^1(X)$. Let  $\vepsilon \searrow 0$, it follows from Corollary~\ref{stabilit1} that $w_\vepsilon \nearrow w \in \cA_m(\omega) \cap C(X)$ uniformly and $w$  satisfies 
\[
	\omega_w^m\wedge \omega^{n-m} 
	\leq {\bf 1}_{\{w = h\}}  F_* \; \omega^n.
\] 
Now we claim that $w = \tilde h$. Indeed, as $w_\vepsilon \leq \tilde h$, it follows that $w \leq \tilde h$. It remains to show that $w \geq \tilde h$ on $\{w < h\}$. Take $v \in SH_m(\omega) \cap L^\infty(X)$ and $v \leq h$. First, we observe that $\omega_w^m \wedge \omega^{n-m} = 0$ on $\{w < v\} \subset \{w < h\}.$ If $\{w< v\}$ were non-empty then by the maximality of $w$  on this set would  give a contradiction (see Theorem~\ref{maximality}, Remark~\ref{maximality-cp}).
\end{proof}

\bigskip

{\noindent Faculty of Mathematics and Computer Science,
Jagiellonian University 30-348 Krak\'ow, \L ojasiewicza 6,
Poland;\\ e-mail: {\tt Slawomir.Kolodziej@im.uj.edu.pl}}\\ \\

{\noindent Faculty of Mathematics and Computer Science,
Jagiellonian University 30-348 Krak\'ow, \L ojasiewicza 6,
Poland;\\ e-mail: {\tt Nguyen.Ngoc.Cuong@im.uj.edu.pl} 

\end{document}